\documentclass[10pt]{article}
\usepackage{flafter,amsmath,amssymb,latexsym,psfrag,graphicx,color,indentfirst}
\usepackage{amsthm}
\usepackage[margin=2.5cm]{geometry}
\usepackage{float}
\usepackage{booktabs}
\usepackage{enumitem} 

\usepackage[colorlinks,
linktocpage=true,
linkcolor=black,
citecolor=black,
urlcolor=black,
anchorcolor=black
]{hyperref}

\usepackage[numbers,sort&compress]{natbib}

\DeclareGraphicsExtensions{.eps,.pdf,.jpg,.png}
\allowdisplaybreaks

\newtheorem{theorem}{Theorem}[section]

\newtheorem{lemma}[theorem]{Lemma}
\newtheorem{remark}[theorem]{Remark}

\newtheorem{example}[theorem]{Example}
\newtheorem{algorithm}[theorem]{Algorithm}
\numberwithin{equation}{section}
\numberwithin{figure}{section}

\begin{document}

\begin{sloppypar}
\setlength\arraycolsep{2pt}
\date{\today}

\title{Efficient Numerical Reconstruction of Wave Equation Sources via Droplet-Induced Asymptotics}
\author{Shutong Hou$^{1}$, Mourad Sini$^{2}$, Haibing Wang$^{1,}$\footnote{Corresponding author, E-mail: hbwang@seu.edu.cn}
\\$^1$School of Mathematics, Southeast University, Nanjing 210096, P.R. China\;
\\$^2$RICAM, Austrian Academy of Sciences, Altenbergerstrasse 69, A-4040, Linz, Austria\,
\\E-mail: 230228431@seu.edu.cn, mourad.sini@oeaw.ac.at, hbwang@seu.edu.cn
}

\maketitle
\begin{abstract}
In this paper, we develop and numerically implement a novel approach for solving the inverse source problem of the acoustic wave equation in three dimensions. By injecting a small high-contrast droplet into the medium, we exploit the resulting wave field perturbation measured at a single external point over time. The method enables stable source reconstructions where conventional approaches fail due to ill-posedness, with potential applications in medical imaging and non-destructive testing. Key contributions include:
\begin{enumerate}
\item Implementation of a theoretically justified asymptotic expansion, from \cite{S-W2022}, using the eigensystem of the Newtonian operator, with error analysis for the spectral truncation.

\item Novel numerical schemes for solving the time-domain Lippmann-Schwinger equation and reconstructing the source via Riesz basis expansions and mollification-based numerical differentiations.

\item Reconstruction requiring only single-point measurements, overcoming traditional spatial data limitations.

\item 3D numerical experiments demonstrating accurate source recovery under noise (SNR of the order $1/a$), with error analysis for the droplet size (of the order $a$) and the number of spectral modes $N$.
\end{enumerate}



\end{abstract}

\section{Introduction}
\label{sec:intro}
In this paper, we address an inverse source problem for the wave equation, focusing on the reconstruction of an unknown source term from external measurements. Let $A_0(x)$ and $B_0(x)$ denote two scalar coefficients modeling the background medium. These coefficients represent, respectively, the inverse of the bulk modulus and the mass density in acoustics. We assume that these coefficients are positive constants, i.e., $A_0(x)\equiv A_0$ and $B_0(x)\equiv B_0$. Define $c_0=(A_0B_0)^{-1/2}$. Furthermore, assume that $J(x,\,t)$ is a source compactly supported in a bounded Lipschitz domain $\Omega \subset \mathbb{R}^3$, with $J=0$ in $\Omega$ for $t<0$. The field $V$ generated by the source $J$ in the homogeneous background medium then satisfies
\begin{equation}
    \left\{
\begin{aligned}
& c_0^{-2}V_{tt}-\Delta V = J & \quad   & \mathrm{in}\  \mathbb{R}^{3} \times (0,\,\mathcal{T}), \\
& V|_{t=0}=0, \quad V_t |_{t=0}=0  & \quad  & \mathrm{in} \ \mathbb{R}^{3}.
\end{aligned}
\right.
\label{eq:V}
\end{equation}

Inverse source problems are inherently ill-posed, and the measurement data available in practical applications are often limited and corrupted by noise. Consequently, reconstructions of the sources pose significant challenges in both theoretical analysis and numerical computation. To obtain accurate and stable reconstructions, we address the problems from two complementary perspectives: modeling and measurement.

On the one hand, from a mathematical standpoint, a sequence of well-posed approximate problems can be constructed to approach the original ill-posed formulation. We then investigate the convergence of these approximate solutions and derive the corresponding error estimates. In the literature, various strategies have been proposed to establish the uniqueness and stability of inverse source problems. By using Carleman estimates and the unique continuation principle, the uniqueness and conditional stability of reconstructing source terms of specific forms from boundary measurement data in bounded domains have been demonstrated \cite{B-Y2017,Klibanov1992}. Using the boundary controllability of the wave equation, the spatial component of a source term can also be uniquely determined from boundary observations \cite{Yamamoto1995}. In unbounded domains, uniqueness results for several special inverse source problems have been obtained using the Laplace transform \cite{H-K-Z2020}. Assuming separability in the spatial and temporal variables of the source term, the paper \cite{C-L-Y2023} proves the uniqueness and conditional stability based on two types of measurements, while the paper \cite{K-L-Y2023} establishes the uniqueness for simultaneously recovering both temporal and spatial components without assuming nonvanishing initial conditions at $t=0$. In the case of a single moving point source, by placing six point sources on the boundary and measuring the corresponding wave field, the uniqueness and Lipschitz stability of the inverse source problem can be established \cite{J-E-T2022}. For a time-dependent semilinear hyperbolic equation, the paper \cite{L-L-X2024} uniquely determines the unknown nonlinearity and/or the unknown sources by using passive or active boundary observations in several generic scenarios. For an inverse source problem involving the stochastic multiterm time-fractional diffusion-wave equation driven by a fractional Brownian motion, the uniqueness is proven for a certain class of functions and the instability is characterized \cite{F-Z-L-W2022}.

On the other hand, from an applied perspective, the instability due to noisy measurement data can be alleviated by enhancing the data itself, often through the integration of targeted mathematical modeling. Specifically, using external intervention methods to enhance the signal of the target can significantly improve the sensitivity of target detection and reduce the impact of noise on the solution of inverse problems. Such interventions include exciting energy into the medium and injecting contrast agents to amplify the signal of the target. In the field of medical imaging, common contrast agents include microbubbles, droplets, nanoparticles, inclusions, and so on, made from special materials; see \cite{B-B-C2011,M-O2014, Qui-al-2009, Quaia-2007, Sheeran-Dayton, I-I-F2018, L-C2015, intro4, dielectric3}  for in depth descriptions and details. The motivation is that the injection of contrast agents modifies the physical parameters of the medium, making their effects more pronounced in the measurements, which in turn enhances the stability and accuracy of the reconstructions. An approach to the mathematical analysis of such imaging modalities can be found in \cite{C-C-M2018,D-G-S2021, G-S2022, G-S-2023, S-S-W-2024, G-S-S-2025, S-S-2025} and the references therein. The main idea there is that by contrasting the fields measured before and after injecting the contrasting agents, one can reconstruct the total fields at the locations of those agents. This is justified in the time harmonic regimes (for both the electromagnetic and acoustic models) when the incident waves are transmitted at frequencies close to the resonating frequencies of the used agents; see \cite{D-G-S2021,G-S2022, G-S-2023} for details. These ideas are extended to the time-domain inverse problems in \cite{S-W2022,S-S-W-2024, S-S-2025}.

In this paper, we reconstruct the unknown source by injecting a high contrast droplet into the medium. Let $D=z+aB$ be a droplet injected into $\Omega$ at the location point $z\in \mathbb{R}^3$, where $a$ is a small parameter and $B\subset \mathbb{R}^3$ is a bounded Lipschitz domain containing the origin. Define
\begin{equation}
    A(x):=\left\{
\begin{aligned}
& A_0 & \quad  & \mathrm{in}\  \mathbb{R}^{3} \setminus D, \\
& A_1 & \quad  & \mathrm{in}\ D,
\end{aligned}
\right.
\label{eq:Ax}
\end{equation}
where $A_1$ is a positive constant with $A_1\sim a^{-2}$ as $a\ll 1$. In acoustics, bubbles with such small values of the bulk and moderate mass density can be designed; see \cite{Z-F2018} for details.

Set
\begin{equation*}
c(x):=(A(x)B_0)^{-1/2}, \quad c_1 :=(A_1 B_0)^{-1/2} \sim a.
\end{equation*}
Define the contrast as
\begin{equation*}
\chi(x):=\frac{c_0^2}{c^2(x)}-1, \quad x\in \mathbb{R}^3.
\end{equation*}
Injecting the droplet $D$ into $\Omega$, the field $U$ generated by $J$ satisfies
\begin{equation}
    \left\{
\begin{aligned}
& c^{-2}(x)U_{tt}-\Delta U = J & \quad   & \mathrm{in}\  \mathbb{R}^{3} \times (0,\,\mathcal{T}), \\
& U|_{t=0}=0, \quad U_t |_{t=0}=0  & \quad  & \mathrm{in} \ \mathbb{R}^{3}.
\end{aligned}
\right.
\label{eq:U}
\end{equation}
The solvability and stability for the direct problems \eqref{eq:V} and \eqref{eq:U} can be found in \cite[Theorem 8.1]{Isakov2017}. Let $W:=U-V$ denote the difference between the perturbed and unperturbed fields. From \eqref{eq:V} and \eqref{eq:U}, we have
\begin{equation}
    \left\{
\begin{aligned}
& c_0^{-2} W_{tt}-\Delta W = -\frac{\chi}{c_0^2} U_{tt} & \quad   & \mathrm{in}\  \mathbb{R}^{3} \times (0,\,\mathcal{T}), \\
& W|_{t=0}=0, \quad W_t |_{t=0}=0  & \quad  & \mathrm{in} \ \mathbb{R}^{3}.
\end{aligned}
\right.
\label{eq:W}
\end{equation}

The inverse problem of our interest is formulated as follows. \\
\textbf{Inverse problem: }Based on the measurement of $U(x^\star,\,t)$ collected at a fixed receiver point $x^\star \in \mathbb{R}^3\setminus \Omega$ over a time interval, with multiple droplets injected one after another, reconstruct the unknown source $J(x,\,t)$ for $(x,\,t)\in \Omega\times(0,\,T)$, where $T< \mathcal{T}$.

We observe that the measured data have $4$ dimensions, $3$ in space, related to the multiple injected droplets, and $1$ in time, while the unknown source function is modeled using $4$ dimensions as well, which makes the inverse problem not overdetermined unlike the traditional inverse source problems. Another key advantage of this approach is that it requires measurements of the field $U$ at only a single fixed point $x^\star \in \mathbb{R}^3\setminus \Omega$, which is particularly appealing in practical applications. 

The novelty of this paper lies in proposing an efficient algorithm for solving the above inverse source problem, namely reconstructing the source term in the wave equation, using the asymptotic expansion method developed in \cite{S-W2022}. To do this, we first use an efficient algorithm for solving the forward problem, namely the Lippmann-Schwinger equation. This is achieved by combining the convolution spline method in time, a high-order volume integral with weak spatial singularity, and spatial interpolations. The resulting method is used to simulate the measurement data. Subsequently, for the inversion procedure, the wave field is reconstructed following the approach in \cite{S-W2022} where a Riesz basis, built up from the eigensystem of the Newtonian operator stated on each used droplet, is used to recover its Riesz series expansion. We provide a rigorous error analysis for the spectral truncation as well as a stability analysis for the reconstruction of the wave field. The source is then recovered using second-order numerical differentiation, which constitutes a mildly ill-posed problem. To address this issue, we introduce a smoothing technique for computing the second-order derivative numerically. The proposed reconstruction method is particularly effective for noisy measurement data and does not rely on any prior assumptions about the data. The efficiency and accuracy of the proposed algorithm are evaluated through numerical experiments.

The remainder of this paper is organized as follows. In Section \ref{sec:Theoretical results on the reconstruction}, we show the theoretical results for the asymptotics of the wave field and the source reconstruction method, where the Newtonian operator is investigated in Appendix. In Section \ref{sec:Numerical algorithms}, we present the corresponding numerical algorithms. Section \ref{sec:Numerical experiments} provides numerical experiments to demonstrate the effectiveness of the proposed algorithms. Finally, conclusions are drawn in Section \ref{sec:Conclusions}.

\section{Theoretical results}
\label{sec:Theoretical results on the reconstruction}

In this section, we present the main theoretical results of this paper, including the asymptotic expansion of the wave field after injecting the droplet, the error estimate introduced by the finite-dimensional truncation, the source reconstruction method and its stability analysis.

\subsection{The droplet-induced asymptotics of the field}

To proceed, we introduce the function space
\begin{equation*}
H_0^\mu (0,\,\mathcal{T}):=\{ g|_{(0,\,\mathcal{T})} : g\in H^\mu(\mathbb{R}) \ \mathrm{with} \ g\equiv 0 \ \mathrm{in} \ (-\infty,\,0) \}, \quad \mu\in \mathbb{R},
\end{equation*}
and generalize it to the $E$-valued function space, denoted by $H_0^\mu (0,\,\mathcal{T};\,E)$ with the Hilbert space $E$. Define
\begin{equation*}
LT(\sigma,\,E):=\{f \in \mathcal{D}^{\prime}_{+}(E): e^{-\sigma t}f \in \mathcal{S}^{\prime}_{+}(E)\},\quad \sigma>0,
\end{equation*}
where $\mathcal{D}^{\prime}_{+}(E)$ and $\mathcal{S}^{\prime}_{+}(E)$ denote the sets of distributions and temperate distributions on $\mathbb{R}$, respectively, with values in $E$ and supports in $[0,\,\infty)$. Then we define the space
\begin{equation*}
H_{0,\,\sigma}^\mu (0,\,\mathcal{T};\,E):=\{ f\in LT(\sigma,\, E) : e^{-\sigma t}\partial_t^k f \in L_0^2(0,\,\mathcal{T};\,E), \, k=1,\,\ldots,\,\mu \}, \quad \mu\in \mathbb{N}^+.
\end{equation*}
For nonnegative integer $\mu$, we introduce the norm
\begin{equation*}
\|f\|_{H_{0,\,\sigma}^\mu (0,\,\mathcal{T};\,E)}:=\Big( \int_0^{\mathcal{T}} e^{-2\sigma t} \Big[ \|f \|^2_E+\sum\limits_{k=1}^{\mu} \mathcal{T}^{2k} \Big\|\frac{\partial^k f}{\partial t^k} \Big\|^2_E \Big] dt \Big)^{1/2}.
\end{equation*}

We assume that the source $J(x,\,t)$ is compactly supported in $\Omega\times [0,\,T]$ and $J \in H_{0,\,\sigma}^p \big( 0,\,T;L^2(\Omega) \big)$, with $p\in \mathbb{N}$, where $\Omega$ is a bounded Lipschitz domain. Note that the solution $V$ can be expressed by a retarded volume potential
\begin{equation}
V(x,\,t):=\int_{\Omega} \frac{J(y,\,t-c_0^{-1}|x-y|)}{4\pi |x-y|} dy.
\label{eq:volume potential J}
\end{equation}
From \cite[Theorem 3.2]{L-M2015}, we have $V \in H_{0,\,\sigma}^p \big( 0,\,+\infty;\,H^1(\mathbb{R}^3) \big)$.

Moreover, define the retarded volume potential $V_D$ by
\begin{equation}
V_D[f](x,\,t):=\int_{D} \frac{f(y,\,t-c_0^{-1}|x-y|)}{4\pi |x-y|} dy, \quad (x,\,t)\in \mathbb{R}^3 \times (0,\,\mathcal{T}).
\end{equation}
From \eqref{eq:W}, we are led to the time-domain Lippmann-Schwinger equation
\begin{equation}
U(x,\,t)+\frac{\chi_1}{c_0^2}V_D[U_{tt}](x,\,t)=V(x,\,t), \quad (x,\,t)\in \mathbb{R}^3 \times (0,\,\mathcal{T}),
\label{eq:Lippmann-Schwinger_equation 1}
\end{equation}
or
\begin{equation}
W(x,\,t)+\frac{\chi_1}{c_0^2}V_D[W_{tt}](x,\,t)=-\frac{\chi_1}{c_0^2} V_D[V_{tt}](x,\,t), \quad (x,\,t)\in \mathbb{R}^3 \times (0,\,\mathcal{T}),
\label{eq:Lippmann-Schwinger_equation 2}
\end{equation}
where
\begin{equation*}
\chi_1 := \frac{c_0^2}{c_1^2} -1 \sim a^{-2} \quad \mathrm{for} \ a\ll 1.
\end{equation*}
For $V \in H_{0,\,\sigma}^{\mu+2} \big(0,\,\mathcal{T};L^2(D) \big)$, by employing the convolution quadrature based argument in \cite{Lubich1994}, we can prove that \eqref{eq:Lippmann-Schwinger_equation 2} has a unique solution in $H_{0,\,\sigma}^\mu \big(0,\,\mathcal{T};L^2(D) \big)$. 

The source reconstruction method presented in this paper is based on the asymptotic expansion of the wave field following the injection of the droplet, utilizing the eigensystem of the Newtonian operator; see Appendix for details of the eigensystem. In the literature, asymptotic expansions of physical fields generated by small droplets are well developed for elliptic models. A heat conduction problem and an acoustic scattering problem with many holes are studied; see \cite{S-W2019,S-W-Y2021} for details. To keep the presentation complete, the asymptotic behavior of the field $U(x,\,t)$ for the wave equation model as $a\to 0$ is given in the following theorem; see \cite{S-W2022} for details.

\begin{theorem}
\label{the:W_expansion} Assume that $J$ is compactly supported in $\Omega \times [0,\,T]$ and $J \in H^6_{0,\,\sigma}(0,\,T;\,L^2(\Omega))$. Then we have the expansion
\begin{equation}
\begin{aligned}
W(x,\,t)  = & -\sum_{n=1}^{+\infty} \frac{1}{4\pi |x-z| \lambda_n} \Big(\int_D e_n(y) dy \Big)^2 V(z,\,t-c_0^{-1}|x-z|)  \\ 
& +\sum_{n=1}^{+\infty} \frac{c_1 \lambda_n^{-\frac{3}{2}}}{4\pi |x-z|} \Big(\int_D e_n(y) dy \Big)^2 \int_{c_0^{-1} |x-z|}^t \sin \Big[\frac{c_1}{\sqrt{\lambda_n}}(t-s) \Big] V(z,\,s-c_0^{-1}|x-z|) ds \\
& + O(a^2), \quad (x,\,t)\in (\mathbb{R}^3\setminus K)\times (0,\,\mathcal{T}),
\end{aligned}
\label{eq:W_expansion for inf}
\end{equation}
where $K$ is any bounded domain in $\mathbb{R}^3$ such that $\overline{D} \subset\subset K$. The expansion holds pointwise in both time and space. Here $\{ \lambda_n, \, e_n \}_{n\in \mathbb{N}^+}$ is the family of the eigenelements of the Newtonian operator $\mathcal{N}: L^2(D)\to L^2(D)$ defined by
\begin{equation}
\mathcal{N}(f)(x):=\int_D \frac{f(y)}{4\pi |x-y|}dy.
\label{eq:Newton operator}
\end{equation}
Note that, the eigenfunctions are normalized, i.e., $\langle e_n,e_n \rangle=1$ for any $n\in \mathbb{N}^+$.
\end{theorem}

\begin{remark}
\label{rem:scale}
From the expansion of the field $W$, the relevant eigenfunctions $e_n$'s are those with nonvanishing averages, i.e., $\int_D e_n(y) dy \neq 0$. From Appendix, we see that
\begin{equation}
\lambda_n \sim a^2, \quad \int_D e_n(y) dy \sim a^{3/2}
\end{equation}
for any fixed $n\in \mathbb{N}^+$. By straightforward calculations, we can easily justify that
\begin{eqnarray}
&& \frac{1}{4\pi |x-z| \lambda_n} \Big(\int_D e_n(y) dy \Big)^2 V(z,\,t-c_0^{-1}|x-z|) = O(a), \nonumber \\
&& \frac{c_1 \lambda_n^{-\frac{3}{2}}}{4\pi |x-z|} \Big(\int_D e_n(y) dy \Big)^2 \int_{c_0^{-1} |x-z|}^t \sin \Big[\frac{c_1}{\sqrt{\lambda_n}}(t-s) \Big] V(z,\,s-c_0^{-1}|x-z|) ds = O(a) \nonumber
\end{eqnarray}
as $a\ll 1$. Hence, the first two terms are dominant in the expansion of $W(x,\,t)$, comparing with the error $O(a^2)$. 
\end{remark}

Next, we give the convergence of the series in \eqref{eq:W_expansion for inf}. In the practical calculation, we only consider the first $N$ terms in the expansion and approximate $W(x,\,t)$ for $(x,\,t)\in (\mathbb{R}^3\setminus K)\times (0,\,\mathcal{T})$ by
\begin{equation}
\begin{aligned}
W_N(x,\,t) := & -\sum_{n=1}^{N} \frac{1}{4\pi |x-z| \lambda_n} \Big(\int_D e_n(y) dy \Big)^2 V(z,\,t-c_0^{-1}|x-z|) \\
& +\sum_{n=1}^{N} \frac{c_1 \lambda_n^{-\frac{3}{2}}}{4\pi |x-z|} \Big(\int_D e_n(y) dy \Big)^2 \int_{c_0^{-1} |x-z|}^t \sin \Big[\frac{c_1}{\sqrt{\lambda_n}}(t-s) \Big] V(z,\,s-c_0^{-1}|x-z|) ds.
\end{aligned}
\label{eq:W_N 1}
\end{equation}
We then provide an error estimate associated with this truncation.

\begin{theorem}
\label{the:convergence of W}
Let $D$ be a spherical droplet. Assuming that $J$ is compactly supported in $\Omega \times [0,\,T]$ and $J \in H^p_{0,\,\sigma}(0,\,T;\,L^2(\Omega))$ with $p\geq 6,\, p\in \mathbb{N}$, the expansion \eqref{eq:W_expansion for inf} is convergent. The truncation error for \eqref{eq:W_N 1} satisfies
\begin{equation}
|W(x,\,t) - W_N(x,\,t)| \lesssim \frac{a}{N},
\label{eq:W_N error}
\end{equation}
where we use the notation \lq\lq $\lesssim$ \rq\rq to denote \lq\lq $\leq$ \rq\rq with its right-hand side multiplied by a generic positive constant.
Furthermore, if $t>T_J+ c_0^{-1}|x-z|$ with $T_J :=T+ c_0^{-1} \mathrm{max}_{y\in \Omega}|z-y|$, then we have
\begin{equation}
|W(x,\,t) - W_N(x,\,t)| \lesssim \frac{a}{N^{p-1}}.
\label{eq:W_N error for the large t}
\end{equation}
Note that, the estimates \eqref{eq:W_N error} and \eqref{eq:W_N error for the large t} hold pointwise in both time and space.
\end{theorem}

\begin{proof}
We consider only the eigenelements with nonvanishing averages.
The normalized eigenfunctions and the corresponding eigenvalues satisfy
\begin{equation}
e_n= \frac{u_{0n0}}{ \sqrt {\langle u_{0n0},\,u_{0n0} \rangle }}, \quad  \lambda_n = \frac{a^2}{\mu_n^2}, \quad n\in \mathbb{N}^+,
\label{eq:e_n and l_n}
\end{equation}
where $\mu_n = (n-\frac{1}{2} )\pi$. For the definition of $u_{0n0}$, refer to \eqref{eq:eigenfunction_nonvanish} and the details in Appendix.

By straightforward calculations, we obtain
\begin{equation}
\Big| \int_D u_{0n0}(y) dy \Big| =4\sqrt{2\pi} \Big(\frac{a}{\mu_n} \Big)^{5/2}, \quad \langle u_{0n0},\,u_{0n0} \rangle  =\frac{8 a^2}{\mu_n} \Big[ \frac{1}{2}-\frac{\sin(2\mu_n)}{4\mu_n} \Big],
\label{eq:int_D for u}
\end{equation}
and
\begin{equation}
\Big|\int_D e_n(y) dy \Big| \lesssim \frac{a^{3/2}}{n^2}, \quad n\in \mathbb{N}^+.
\label{eq:estimate of int_en}
\end{equation}
Denote
\begin{equation}
    \xi_n(x,\,t):= -\frac{1}{4\pi |x-z| \lambda_n} \Big(\int_D e_n(y) dy \Big)^2 V(z,\,t-c_0^{-1}|x-z|), \quad n\in \mathbb{N}^+.
    \label{eq:xi_n}
\end{equation}
From \eqref{eq:e_n and l_n} and \eqref{eq:estimate of int_en}, we deduce
\begin{equation}
    |\xi_n(x,\,t)|  \lesssim \frac{a}{n^2}.
\end{equation}

Next, we estimate 
\begin{equation}
\zeta_n(x,\,t) :=\frac{c_1 \lambda_n^{-\frac{3}{2}}}{4\pi |x-z|} \Big(\int_D e_n(y) dy \Big)^2 \int_{c_0^{-1} |x-z|}^t \sin \Big[\frac{c_1}{\sqrt{\lambda_n}}(t-s) \Big] V(z,\,s-c_0^{-1}|x-z|) ds, \quad n\in \mathbb{N}^+. 
\label{eq:zeta_n}
\end{equation}
Since $J \in H^p_{0,\,\sigma}(0,\,T;\,L^2(\Omega))$, it follows that $V \in H_{0,\,\sigma}^p \big(0,\,T_J;\,H^1(\Omega) \big)$; see \cite[Theorem 3.2]{L-M2015} for details. Moreover, owing to $\Delta V=J - c_0^{-1}V_{tt}$, we see that $V\in  H_{0,\,\sigma}^{p-2} \big(0,\,T_J;\,H^2(\Omega) \big),\ p\geq 6$. Applying the Sobolev embedding theorem in space yields
\begin{equation}
V(\cdot,\,t) \in C^{0,\,1/2} (\overline{\Omega}) \subset C(\overline{\Omega}), \quad t \in [0,\,T_J],
\label{eq:boundness of V pointwise}
\end{equation}
which holds pointwise in time. Applying the Sobolev embedding theorem in time, we have $V(z,\,\cdot) \in C([0,\,T_J])$ as $V \in H_{0,\,\sigma}^p \big(0,\,T_J;\,H^1(\Omega) \big),\ p\geq 6$. Consequently, $|V(z,\,\cdot)|$ is bounded for $t\in [0,\,T_J]$. Since $c_1 \sim a$, we observe from \eqref{eq:e_n and l_n} that $\frac{c_1}{\sqrt{\lambda_n}} \sim n$. Then, we obtain
\begin{equation}
    \Big| \int_{c_0^{-1} |x-z|}^t \sin \Big[\frac{c_1}{\sqrt{\lambda_n}}(t-s) \Big] V(z,\,s-c_0^{-1}|x-z|) ds \Big|  \lesssim 
    \Big| \int_{c_0^{-1} |x-z|}^{T_J+ c_0^{-1}|x-z|} \sin \Big[\frac{c_1}{\sqrt{\lambda_n}}(t-s) \Big] ds \Big| \lesssim \frac{1}{n}.
\label{eq:integral of sin_V}
\end{equation}
Combining \eqref{eq:e_n and l_n}, \eqref{eq:estimate of int_en}, and \eqref{eq:integral of sin_V}  yields
\begin{equation}
\label{eq:zeta_n_es}
|\zeta_n(x,\,t)|  \lesssim \frac{a}{n^2}.
\end{equation}
As the series $\{ \frac{1}{n^2} \}_{n\in \mathbb{N}^+}$ is convergent, the expansion \eqref{eq:W_expansion for inf} is convergent, with the truncation error given by \eqref{eq:W_N error}.

Furthermore, for $t>T_J+ c_0^{-1}|x-z|$, we have
\begin{equation}
\partial_t^k V(z,\,t-c_0^{-1}|x-z|)=0,  \quad k=0,\,1,\,\ldots,\,p,
\label{eq:partial tk V}
\end{equation}
and the term $\xi_n$ vanishes. From \eqref{eq:partial tk V}, by using integration by parts, we obtain
\begin{equation*}
\begin{aligned}
& \int_{c_0^{-1} |x-z|}^t \sin \Big[\frac{c_1}{\sqrt{\lambda_n}}(t-s) \Big] V(z,\,s-c_0^{-1}|x-z|) ds \\
=& \frac{\sqrt{\lambda_n}}{c_1} \Big( \cos \Big[\frac{c_1}{\sqrt{\lambda_n}}(t-s) \Big] V(z,\,s-c_0^{-1}|x-z|) \Big|_{c_0^{-1} |x-z|}^t  \\
&  \quad\quad\quad - \int_{c_0^{-1} |x-z|}^t \cos \Big[\frac{c_1}{\sqrt{\lambda_n}}(t-s) \Big] \partial_t V(z,\,s-c_0^{-1}|x-z|) ds \Big)  \\
=& \Big(-\frac{\sqrt{\lambda_n}}{c_1} \Big) \int_{c_0^{-1} |x-z|}^t \cos \Big[\frac{c_1}{\sqrt{\lambda_n}}(t-s) \Big] \partial_t V(z,\,s-c_0^{-1}|x-z|) ds \\
=& \Big(-\frac{\sqrt{\lambda_n}}{c_1} \Big) \int_{c_0^{-1} |x-z|}^{T_J+ c_0^{-1}|x-z|} \cos \Big[\frac{c_1}{\sqrt{\lambda_n}}(t-s) \Big] \partial_t V(z,\,s-c_0^{-1}|x-z|) ds.
\end{aligned}
\end{equation*}
By the similar argument in deriving \eqref{eq:zeta_n_es}, we have
\begin{equation*}
\partial_t^k V(\cdot,\,t) \in C^{0,\,1/2} (\overline{\Omega}) \subset C(\overline{\Omega}), \quad t \in [0,\,T_J], \quad k=0,\,1,\,\ldots,\,p-2, 
\end{equation*}
and 
\begin{equation*}
\partial_t^k V(z,\,\cdot) \in C([0,\,T_J]), \quad z\in \Omega, \quad k=0,\,1,\,\ldots,\,p-2.  
\end{equation*}
Then, it follows that
\begin{equation}
\Big| \int_{c_0^{-1} |x-z|}^t \sin \Big[\frac{c_1}{\sqrt{\lambda_n}}(t-s) \Big] V(z,\,s-c_0^{-1}|x-z|) ds \Big| \lesssim \frac{1}{n^{p-1}}.
\label{eq:estimate of integral for large t}
\end{equation}
Combining \eqref{eq:e_n and l_n}, \eqref{eq:estimate of int_en}, and \eqref{eq:estimate of integral for large t} yields
\begin{equation}
|\zeta_n(x,\,t)|  \lesssim \frac{a}{n^{p}},
\end{equation}
thus improving the estimate \eqref{eq:W_N error} to \eqref{eq:W_N error for the large t}.
\end{proof}

\subsection{Inversion formulas using the droplet-induced asymptotics}

From Theorem~\ref{the:W_expansion}, the dominant term in the asymptotic expansion of $W$ is expressed as an infinite series in terms of the eigensystem of the Newtonian operator. Due to the compactness of the support of $J$, we can derive that the field $V(z,\,t)$ is compactly supported with respect to $t$ in $[0,\,T_J]$. Then, for $t>T_J+c_0^{-1}|x-z|$, the integrals in the expansion of $W$ are independent of $t$ and we have $W(x,\,t)=U(x,\,t),\,x\in\mathbb R^3\setminus\Omega$. Meanwhile, the family $\{ \sin(\omega_n t),\, \cos(\omega_n t) \}_{n\in \mathbb{N}^+}$ with $\omega_n:=\frac{c_1}{\sqrt{\lambda_n}}$ forms a Riesz basis. By using the Riesz theory and shifting the time interval, the wave field $V(z,\,t)$ is reconstructed from the measurement of $U$, as stated in the following theorem from \cite{S-W2022}. Note that, due to the simplified eigensystem presented in Appendix, the resulting expression is more concise than that in \cite{S-W2022}.

\begin{theorem}
\label{the:inverse sourece}
Let $D$ be a spherical droplet. Assume that $J$ is compactly supported in $\Omega \times [0,\,T]$ and $J \in H^6_{0,\,\sigma}(0,\,T;\,L^2(\Omega))$. For $x \in \mathbb{R}^3 \setminus \Omega$, we measure $U(x,\,t)$ for an interval of the form $[\tilde{T},\,\tilde{T}+2\pi/b]$ with $\tilde{T}> T_J +c_0^{-1}|x-z|$ and $b:=c_1\pi/a$. Then 
\begin{equation}
V(z,\,t)=\sum_{n=1}^{+\infty} \frac{b}{\pi}[ C_n(x,\,z)g_n(t-\pi/b) + D_n(x,\,z) h_n(t-\pi/b)], \quad t\in [0,\,2\pi/b],
\label{eq:V_rec_expansion}
\end{equation}
with 
\begin{eqnarray}
&& C_n(x,\,z):= \int_0^{2\pi/b} \cos(\omega_n (s-\pi/b)) V(z,\,s)ds, \quad n\in \mathbb{N}^+, \label{eq:C_n from V} \\
&& D_n(x,\,z):= \int_0^{2\pi/b} \sin(\omega_n (s-\pi/b)) V(z,\,s)ds, \quad n\in \mathbb{N}^+, \label{eq:D_n from V}
\end{eqnarray}
where $g_n(t):=(\mathcal{F}^{-1})^\star \cos((n-1)t)=\cos(\omega_n t)$ and $h_n(t):=(\mathcal{F}^{-1})^\star \sin(nt)=\sin(\omega_n t)$. The operator $\mathcal{F}$ is the one linking the Riesz basis $\{ \sin(\omega_n t),\, \cos(\omega_n t) \}_{n\in \mathbb{N}^+}$ to the canonical basis $\{ \sin(n t),\, \cos((n-1) t) \}_{n\in \mathbb{N}^+}$, with $\omega_n=b(n-\frac{1}{2})$. The convergence of \eqref{eq:V_rec_expansion} holds pointwise in both time and space. The coefficients $C_n(x,\,z)$ and $D_n(x,\,z)$ can be estimated from the data $U(x,t)$ as follows:
\begin{eqnarray}
&& C_n(x,\,z)=\cos \big(\omega_n (\tilde{T}-c_0^{-1}|x-z|) \big)A_n(x,\,z) - \sin\big( \omega_n (\tilde{T}-c_0^{-1}|x-z|) \big) B_n(x,\,z), \label{eq:C_n for reconstruction} \\
&& D_n(x,\,z)=\sin \big(\omega_n (\tilde{T}-c_0^{-1}|x-z|) \big)A_n(x,\,z) + \cos\big(\omega_n (\tilde{T}-c_0^{-1}|x-z|) \big) B_n(x,\,z), \label{eq:D_n for reconstruction}
\end{eqnarray}
where 
\begin{eqnarray}
&& A_n(x,\,z)=\frac{b}{\pi\alpha_n} \int_0^{2\pi/b} h_n(s-\pi/b) U(x,\,\tilde{T}+s) ds + O(na), \quad n\in \mathbb{N}^+, \label{eq:A_n for reconstruction} \\
&& B_n(x,\,z)=-\frac{b}{\pi\alpha_n} \int_0^{2\pi/b} g_n(s-\pi/b) U(x,\,\tilde{T}+s) ds + O(na), \quad n\in \mathbb{N}^+, \label{eq:B_n for reconstruction}
\end{eqnarray}
with
\begin{equation}
\alpha_n:=\alpha_n(x,\,z):=\frac{\omega_n (\int_D e_n(y) dy)^2}{4\pi |x-z|\lambda_n}.
\label{eq:alpha_n}
\end{equation}
\end{theorem}

\begin{remark}
\label{rem:b and c1}
From \eqref{eq:V_rec_expansion}, if $T\leq 2\pi/b$, the field $V$ can be reconstructed over the time interval $[0,\,T]$. Moreover, the parameter $b$ is given by $b=c_1\pi/a=\frac{\pi}{a}(A_1 B_0)^{-1/2}$, where $A_1$ depends on the material properties of the droplet $D$. This provides a guidance for the design of the droplet in the source reconstruction.
\end{remark}

\subsection{Stability estimates of the inversion formulas}
\label{subsec:Stability estimates of the inversion formulas}
Based on Theorem~\ref{the:inverse sourece}, we analyze the stability of the reconstruction of the field $V$ for noisy measurement data. To minimize the reconstruction error, the choice of the truncation number depends on both the noise and the size of the droplet.

\begin{theorem}
\label{the:Stability estimates of the inversion formulas}
Under the assumptions in Theorem~\ref{the:convergence of W}, we consider the noisy measurement data
\begin{equation*}
|U^{\delta}(x,\,t) - U(x,\,t)| \leq \delta, \quad x \in \mathbb{R}^3 \setminus \Omega,\,t\in [\tilde{T},\,\tilde{T}+2\pi/b],
\end{equation*}
where $\delta >0$ represents the level of the measurement noise. Define 
\begin{eqnarray*}
&& A_n^\prime(x,\,z):=\frac{b}{\pi\alpha_n} \int_0^{2\pi/b} h_n(s-\pi/b) U^\delta(x,\,\tilde{T}+s) ds, \nonumber \\
&& B_n^\prime(x,\,z):=-\frac{b}{\pi\alpha_n} \int_0^{2\pi/b} g_n(s-\pi/b) U^\delta(x,\,\tilde{T}+s) ds,  \nonumber \\
&& C_n^\prime(x,\,z):=\cos \big(\omega_n (\tilde{T}-c_0^{-1}|x-z|) \big)A_n^\prime(x,\,z) - \sin\big( \omega_n (\tilde{T}-c_0^{-1}|x-z|) \big) B_n^\prime(x,\,z), \nonumber \\
&& D_n^\prime(x,\,z):=\sin \big(\omega_n (\tilde{T}-c_0^{-1}|x-z|) \big)A_n^\prime(x,\,z) + \cos\big(\omega_n (\tilde{T}-c_0^{-1}|x-z|) \big) B_n^\prime(x,\,z),
\end{eqnarray*}
and set
\begin{equation}
V_N(z,\,t):=\sum_{n=1}^{N} \frac{b}{\pi}[ C_n^\prime(x,\,z)g_n(t-\pi/b) + D_n^\prime(x,\,z) h_n(t-\pi/b)], \quad t\in [0,\,2\pi/b].
\label{eq:VN_truncation}
\end{equation}
Then, if $T_J \leq 2\pi/b$ or if $V$ is compactly supported in $\Omega \times [0,\,2\pi/b]$, it follows that
\begin{equation}
    | V(x,\,t)-V_N(x,\,t) |  \leq \frac{b}{\pi} \Big[ 2 \tilde{C}_1 N^2 \Big(\frac{\delta}{a}+a \Big)+ \frac{\tilde{C}_2}{N^{p-2}} \Big], \quad p\geq 6,
    \label{eq:V minnus V_N2}
\end{equation}
where $\tilde{C}_1$ and $\tilde{C}_2$ are positive constants independent of $N$.
\end{theorem}

\begin{proof}
Note that $\alpha_n \sim a/n$. From \eqref{eq:A_n for reconstruction} and \eqref{eq:B_n for reconstruction}, there exists a constant $\tilde{C}_1$ such that
\begin{equation}
    \big| C_n(x,\,z)-C_n^\prime(x,\,z) \big| \leq \tilde{C}_1 n (\delta/a+a), \quad \big| D_n(x,\,z)-D_n^\prime(x,\,z) \big| \leq \tilde{C}_1 n (\delta/a+a), \quad n\in \mathbb{N}^+.
    \label{eq:C_n and D_n and C_1 a}
\end{equation}
Since the Riesz basis expansion in \eqref{eq:V_rec_expansion} is convergent, there exists a constant $\tilde{C}_2$ such that 
\begin{equation}
    \Big| \sum_{n=N+1}^{+\infty} \big[ C_n(x,\,z)g_n(t-\pi/b) + D_n(x,\,z) h_n(t-\pi/b) \big] \Big| \leq \tilde{C}_2.
    \label{eq:C_n plus D_n}
\end{equation}
If $T_J \leq 2\pi/b$ or if $V$ is compactly supported in $\Omega \times [0,\,2\pi/b]$, by using the similar argument in deriving \eqref{eq:estimate of integral for large t}, the truncation error \eqref{eq:C_n plus D_n} can be improved by 
\begin{equation*}
    \Big| \sum_{n=N+1}^{+\infty} \big[ C_n(x,\,z)g_n(t-\pi/b) + D_n(x,\,z) h_n(t-\pi/b) \big] \Big| \leq \frac{\tilde{C}_2}{N^{p-2}}, \quad p\geq 6.
\end{equation*}
By straightforward calculations, we have
\begin{equation*}
\begin{aligned}
| V(x,\,t)-V_N(x,\,t) |  \leq & \frac{b}{\pi} \Big| \sum_{n=1}^{N} \big[ \big(C_n(x,\,z)-C_n^\prime(x,\,z)\big)g_n(t-\pi/b) + \big(D_n(x,\,z)-D_n^\prime(x,\,z)\big) h_n(t-\pi/b) \big] \Big|  \\
& + \frac{b}{\pi} \Big| \sum_{n=N+1}^{+\infty}  \big[ C_n(x,\,z)g_n(t-\pi/b) + D_n(x,\,z) h_n(t-\pi/b) \big] \Big| \\
 \leq & \frac{b}{\pi} \Big[ 2\tilde{C}_1 N^2 \Big(\frac{\delta}{a}+a \Big)+ \frac{\tilde{C}_2}{N^{p-2}} \Big].
\end{aligned}
\end{equation*}
The proof is complete.
\end{proof}

\begin{remark}
\label{rem:stability analysis}
We end this section with the two remarks.
\begin{enumerate}
\item If we take $p=6$ and choose $N\sim 1/(\delta/a+a)^{1/6}$, then the reconstruction error \eqref{eq:V minnus V_N2} becomes $O\big( (\delta/a+a)^{2/3} \big)$. Therefore, if, in particular, the noise level satisfies $\delta=O(a^2)$, by choosing $N\sim 1/a^{1/6}$, the reconstruction error is $O(a^{2/3})$.

\item If the noise level is of the order $\delta=O(a^2)$, then the coefficients $A_n(x,\,z)$ and $B_n(x,\,z)$ can still be computed with an error $O(a)$, which represents the optimal accuracy achievable by this method. In this case, the Signal-to-Noise Ratio (SNR) is $O(1/a)$, since the measurement data $U(x,\,t)=W(x,\,t)= O(a)$ for $t>\tilde{T}$. From Remark~\ref{rem:b and c1}, we can design a droplet that satisfies $T_J \leq 2\pi/b$. Also, we can reconstruct the source more accurately with the decrease of the size of the droplet, however, with the stricter requirement of the measurement data or instruments that ensure noise levels of the order $O(a^2)$. To reconcile these facts, we need to find a balance for the size of the droplet. Moreover, from \eqref{eq:W_N error} or \eqref{eq:W_N error for the large t}, by choosing $N \geq 1/a$ or $N \geq 1/a^{1/(p-1)}$, the error caused by the truncation in \eqref{eq:W_expansion for inf} will be covered by the measurement noise, which is useful for generating simulation data.
\end{enumerate}
\end{remark}

\section{Numerical algorithms}
\label{sec:Numerical algorithms}
In this section, we focus on numerical schemes for reconstructing the source. To simulate the field $U$, we propose an efficient algorithm to solve the Lippmann-Schwinger equation \eqref{eq:Lippmann-Schwinger_equation 1}. Then, based on Theorem \ref{the:inverse sourece}, we develop an algorithm to reconstruct the field $V$ in the domain $\Omega$ with a small error. Consequently, we reconstruct the source $J$ by numerically differentiating the reconstructed field $V$. 

\subsection{An efficient algorithm for solving the Lippmann-Schwinger equation}
\label{subsec:algorithm for solving the Lippmann-Schwinger equation}
In this subsection, we develop an efficient algorithm to numerically solve the forward problem, which is employed to simulate the measurement data. The forward problem is formulated as the time-domain Lippmann-Schwinger equation \eqref{eq:Lippmann-Schwinger_equation 1}. Using convolution quadrature in time and trigonometric collocation in space, the paper \cite{L-M2015} computes an approximate solution and provides numerical results in two spatial dimensions. In contrast, we solve the time-domain Lippmann-Schwinger equation using the convolution spline method \cite{D-D2014,B-G-H2020} in time, combined with a high-order scheme for evaluating the weakly singular volume integral in space.

For simplicity, we assume that the droplet $D$ is spherical. Our algorithm can be extended to the general case with modifications to the shapes of $D$. It consists of three steps.

(i) We approximate the field in time using the convolution spline method. A uniform sequence of time steps $t_l=l \Delta t \in [0,\,\mathcal T],\,l=0,\,1,\,\ldots,\,\lfloor \mathcal T/\Delta t \rfloor$ is introduced for temporal discretization, where $\Delta t$ denotes the time step size and $\lfloor x \rfloor$ indicates the greatest integer less than or equal to $x$. We adopt the convolution spline as described in \cite{B-G-H2020}.
Let $q$ be an integer and $\{ \tau_{m},\,m=0,\,1,\,\ldots \}$ be the uniform grid on the interval $[0,\,\mathcal T)$. Let $S_{m}$ be the set of $2q+1$ nearest nodes for $\tau_{m}$. For $m<q$, $S_{m}=\{\tau_0,\,\ldots,\,\tau_{2q} \}$, and for $m\geq q$, $S_{m}=\{\tau_{m-q},\,\ldots,\,\tau_{m+q}\}$. Given the data $\{ y_{m},\, m=0,\,1,\,\ldots\}$ corresponding to $\{ \tau_{m},\ m=0,\,1,\,\ldots\}$, let $P_{m}(\tau)$ be the Lagrange interpolation polynomial of degree $2q$ defined by the data $\{ y_{j}\}$ on the stencil $S_{m}$, that is, $P_{m}(\tau_{j})=y_{j}$ for all $j$ with $\tau_{j} \in S_{m}$. On the interval $(\tau_{m},\,\tau_{m+1})$, the D-spline $D_{m+1/2}(\tau)$ is defined as the degree-$(4q+1)$ Hermite interpolation of $P_{m}(\tau)$ and $P_{m+1}(\tau)$. Precisely, for each $m$, $D_{m+1/2}(\tau)$ satisfies
\begin{equation*}
\frac{d^{\ell} D_{m+1/2}}{dt^{\ell}}(\tau_{m})= \frac{d^{\ell} P_{m}}{dt^{\ell}}(\tau_{m}), \quad
\frac{d^{\ell} D_{m+1/2}}{dt^{\ell}}(\tau_{m+1})= \frac{d^{\ell} P_{m+1}}{dt^{\ell}}(\tau_{m+1}), \quad \ell=0,\,\ldots,\,2q.
\end{equation*}
As the interpolation operators are linear and $D_{m+1/2}(\tau)$ depends only on the data $\{ y_{j}\}$ on $S_{m}\cup S_{m+1}$, there exist degree-$(4q+1)$ polynomials $\tilde\omega_{m+1/2,\,j}(\tau),\ \tau_{j}\in S_{m}\cup S_{m+1}$, such that
\begin{equation*}
D_{m+1/2}(\tau)=\sum\limits_{\{j:\,\tau_{j}\in S_{m}\cup S_{m+1}\}} y_{j} \tilde{\omega}_{m+1/2,\,j}(\tau),\ \ \tau\in (\tau_{m},\,\tau_{m+1}).
\end{equation*}
Let $m(\tau)$ be the index of the rightmost node not larger than $\tau$, i.e., such that $\tau_{m(\tau)}\leq \tau < \tau_{m(\tau)+1}$. Then the convolution spline $\omega_{j}(\tau)$ associated with the node $j$ and $\tau\in (0,\,T)$ is defined as
\begin{equation*}
\omega_{j}(\tau) = \left\{
\begin{aligned}
& \tilde{\omega}_{m(\tau)+1/2,\,j}(\tau), & \ \ &  \mathrm{if}\ j \  \mathrm{is \ such \ that \ } \tau_{j} \in S_{m}\cup S_{m+1},   \\
& 0, & \ \ & \mathrm{otherwise}.
\end{aligned}
\right.
\end{equation*}
These convolution splines are piecewise polynomials with compact supports and satisfy $\omega_{j}(\tau) \in C^{2q} [0,\,\mathcal{T}]$.

Then, for the field $U$, our interpolation algorithm takes the following form:
\begin{eqnarray}
&& U(y,\,t_{l}-\tau) \approx \sum\limits_{s=0}^{\infty} \omega_s (\tau) U(y,\,t_{l}-t_{s}), \quad l=1,\,\ldots,\,\lfloor \mathcal{T}/\Delta t \rfloor, \label{eq:intp1}\\
&& \frac{\partial^{2} U(y,\, t_{l}-\tau)}{\partial t^2} \approx \sum\limits_{s=0}^{\infty} \omega_{s}^{\prime\prime} (\tau) U(y,\,t_{l}-t_{s}), \quad l=1,\,\ldots,\,\lfloor \mathcal{T}/\Delta t \rfloor,  \label{eq:intp2}
\end{eqnarray}
where the spline function $\omega_s$ exhibits different forms and regularities depending on the order $q\in \mathbb{N}^+$. In this context, $n_q := \lceil\mathrm{diam}(D)/(c_0 \Delta t)\rceil +q$ is sufficiently large to capture all historical dependencies in the approximations \eqref{eq:intp1} and \eqref{eq:intp2} as inferred from Huygens' principle and the support of the convolution splines. Here, $\lceil x \rceil$ denotes the smallest integer greater than or equal to $x$ and $\mathrm{diam}(D)$ denotes the diameter of $D$.

(ii) We compute the volume integral with the weak singularity in \eqref{eq:Lippmann-Schwinger_equation 1}. For a smooth function $f: D \times (0,\,\mathcal T) \to \mathbb{R}$, we rewrite the integral in polar coordinates as
\begin{eqnarray}
\int_{D} \frac{f(y,\,t-c_0^{-1}|x-y|)}{4\pi |x-y|} dy && =\int\limits_{ \{(r,\,\theta,\,\varphi): \ y(r,\,\theta,\,\varphi)\in D \} } \frac{f(y,\,t-c_0^{-1}r)}{4\pi}  r \sin(\theta) dr d\theta d\varphi, \nonumber \\
&& =\int_0^1 \int_{0}^{\pi} \int_{0}^{2\pi} \frac{f \big( y(r^\prime,\theta,\varphi),\,t-c_0^{-1}r^{\prime}r_b(\theta,\,\varphi) \big)}{4\pi}  r^{\prime}r^2_b(\theta,\,\varphi) \sin(\theta) dr^{\prime} d\theta d\varphi, \label{eq:volume integration 2}
\end{eqnarray}
where $r=|x-y|=r^{\prime} r_b(\theta,\,\varphi)$, and $\theta,\,\varphi$ are parameters in polar coordinates with the origin at $x\in D$. Here, $r^{\prime}\in [0,\,1)$ and $r_b(\theta,\,\varphi) :=\sqrt{C_x^2+(a^2-|x|^2)}-C_x$ with $C_x:=x \cdot \big( \sin(\theta)\cos(\varphi),\,\sin(\theta)\cos(\varphi),\,\cos(\theta) \big)$. In this manner, the volume integral containing a weak singularity is converted into a regular volume integral in the unit ball.

To compute the integral \eqref{eq:volume integration 2}, we consider a spatial discretization in the unit ball. This involves combining discretizations on the surface and in the radial direction. Let $-1<\tilde{\beta}_{1}<\tilde{\beta}_{2}<\ldots<\tilde{\beta}_{N_s}<1$ be the zeros of the Legendre polynomial $P_{N_s}$ of order $N_s$ and let
\begin{equation}
\tilde{\alpha}_{j}:= \frac{2(1-\tilde{\beta}_{j}^{2})}{[N_s P_{N_s-1}(\tilde{\beta}_{j})]^{2}},\quad  j=1,\,\ldots,\,N_s,
\label{eq:alpha_j}
\end{equation}
denote the weights of the Gauss-Legendre quadrature rule. Then we choose a set of points $\hat{x}_{jk}$ on the unit sphere $\mathbb{S}^{2}$, defined by $\hat{x}_{jk}:=(\sin \theta_{j}\cos \phi_{k},\, \sin \theta_{j} \sin \phi_{k},\, \cos \theta_{j})$, $j=1,\,\ldots,\,N_s,\,k=0,\,\ldots,\,2N_s-1$,
where $\theta_{j}:=\arccos \tilde{\beta}_{j}$ and $\phi_{k}:=\pi k/N_s$. The quadrature of a smooth function $f:\mathbb{S}^{2}\to \mathbb{R}$ is approximated by the so-called Gauss trapezoidal rule
\begin{equation}
\int_{\mathbb{S}^2} f ds \approx \frac{\pi}{N_s} \sum_{j=1}^{N_s} \sum_{k=0}^{2N_s-1} \tilde{\alpha}_{j}f(\hat{x}_{jk}).
\end{equation}
The error associated with this quadrature decreases exponentially as $N_s$ increases; see \cite{Wienert1990} for details. 
Along the radial direction, we apply an $N_r$-node Gauss-Legendre rule. Let $r_\iota,\,\iota=1,\ldots,\,N_r$ be the nodes and $\eta_\iota,\,\iota=1,\ldots,\,N_r$ be the corresponding wights for the interval $[0,1]$. Therefore, there are $\overline{N}=2N_r N_s^2$ nodes for spatial discretization. The three-dimensional unit ball $B(0,\,1)$ can be parameterized by a function $x=\gamma(r^\prime,\,\theta,\,\varphi)$, where $\gamma: [0,\,1]\times[0,\,\pi]\times[0,\,2\pi] \to B(0,\,1)$. Under the map $\gamma$, we have
$x_i=r_\iota (\sin \theta_{j}\cos \varphi_{k},\, \sin \theta_{j} \sin \varphi_{k},\, \cos \theta_{j})$, $i=1,\,\ldots,\,\overline{N}$,
where $\iota,\,j,\,k$ are the indices corresponding to the directions $r^\prime,\,\theta,\,\varphi$ for the node $i$.

(iii) We consider a spatial interpolation of the field within the unit ball. For each point $(r^\prime,\,\theta,\,\phi)$, we select $2n_0+1$ nearest nodes, where $n_0 \in \mathbb{N}^{+}$ and $2n_0+1 \leq \min\{N_r,\,N_s\}$. Consequently, there are $(2n_0+1)^3$ nearest nodes surrounding the point $x$, whose indices form the set $P_x$. The interpolation is then performed using these nearest nodes as follows:
\begin{equation}
f(x)\approx \sum_{j \in P_x} f(x_j) L_j(x), \quad L_j(x) :=L_j^r(r^\prime) L_j^\theta(\theta) L_j^\phi(\phi),
\label{eq:interpolation_3d}
\end{equation}
where $L_j^r,\,L_j^\theta,\,L_j^\phi$ are the standard one-dimensional Lagrange polynomials corresponding to the nearest $2n_0+1$ Legendre nodes in their respective coordinates. For any node $j \notin P_x$, the interpolation weight satisfies $L_j(x)=0$.

We are now prepared to solve the Lippmann–Schwinger equation \eqref{eq:Lippmann-Schwinger_equation 1} following the three steps outlined above. We first solve the equation for $x\in D$, and then the solution for $x\in \mathbb{R}^3\setminus \overline{D}$ follows directly. Define a mapping $\mathcal{G}: B(0,\,1) \to D$ by
$\mathcal{G}(x):=z+ax,\, x\in B(0,\,1)$. Let $U(x,\,t_l):=U^l(x),\  l=1,\,\ldots,\,\lfloor \mathcal T/\Delta t \rfloor$. Applying the above discretizations and interpolation to the Lippmann–Schwinger equation \eqref{eq:Lippmann-Schwinger_equation 1}, we find
\begin{equation}
\begin{aligned}
U^l \big( \mathcal{G}(x_n) \big) & +\frac{\chi_1 a^2}{c_0^2} \sum_{\iota=1}^{N_r} \sum_{j=1}^{N_s} \sum_{k=0}^{2N_s-1} \frac{\eta_\iota \tilde{\alpha}_j}{4N_s} \omega_0 \big(c_0^{-1} r_\iota r_b(\theta_j,\,\varphi_k) \big) r_\iota \Big[\frac{r_b(\theta_j,\,\varphi_k)}{a}\Big]^2  \sum_{i^\prime=1}^{\overline{N}} L_{i^\prime} \big(\mathcal{G}(x_i) \big) U^l \big( \mathcal{G}(x_{i^\prime}) \big)  \\
& =\tilde{V} \big( \mathcal{G}(x_n),\,t_l \big), \quad  n=1,\,\ldots,\,\overline{N}, \ l=1,\ldots,\,\lfloor \mathcal{T}/\Delta t \rfloor,
\end{aligned}
\label{eq:time-lpm discretization}
\end{equation}
where $\iota,\,j,\,k$ are the indices corresponding to the directions $r^\prime,\,\theta,\,\varphi$ for the node $i$, and
\begin{eqnarray*}
&& \tilde{V} \big( \mathcal{G}(x_n),\,t_l \big) :=V \big(\mathcal{G}(x_n),\,t_l \big)  -\frac{\chi_1 a^2}{c_0^2} \sum\limits_{l^\prime=1}^{\mathrm{min}(n_q,\,l-1)} Q_{l,\,l^\prime}, \\
&& Q_{l,\,l^\prime}:=\sum_{\iota=1}^{N_r} \sum_{j=1}^{N_s} \sum_{k=0}^{2N_s-1} \frac{\eta_\iota \tilde{\alpha}_j}{4N_s} \omega_{l^\prime} \big(c_0^{-1} r_\iota r_b(\theta_j,\,\varphi_k) \big) r_\iota \left[\frac{r_b(\theta_j,\,\varphi_k)}{a}\right]^2  \sum_{i^\prime=1}^{\overline{N}} L_{i^\prime} \big( \mathcal{G}(x_i) \big) U^{l-l^\prime} \big( \mathcal{G}(x_{i^\prime}) \big).
\end{eqnarray*}
Then, from \eqref{eq:Lippmann-Schwinger_equation 1} and \eqref{eq:intp2}, the field $U(x,\,t)$ for $x\in \mathbb{R}^3\setminus \overline{D}$ is approximated by
\begin{equation}
    U(x,\,t)\approx V(x,\,t) - \frac{\chi_1 a^3}{c_0^2} \sum_{\iota=1}^{N_r} \sum_{j=1}^{N_s} \sum_{k=0}^{2N_s-1} \frac{\eta_\iota \tilde{\alpha}_j}{4N_s} r_\iota^2
    \sum_{l^\prime=0}^{\mathrm{min}(n_q,\,l-1)} \omega_{l^\prime}(d_i) \frac{ U^{n_i-l^\prime} \big(\mathcal{G}(x_i) \big)}{ \big|x-\mathcal{G}(x_i) \big| },
\end{equation}
where $n_i := \lfloor \big(t-c_0^{-1} |x-\mathcal{G}(x_i)| \big)/\Delta t \rfloor +1$ and $0< d_i :=n_i \Delta t- \big(t-c_0^{-1} |x-\mathcal{G}(x_i)| \big) \leq \Delta t$.

\subsection{An efficient algorithm for reconstructing the source}
In this subsection, we propose an algorithm for reconstructing the wave field $V$ and the source term $J$.
\begin{algorithm}
\label{alg:invese_source}
Assume that $J$ is compactly supported in $\Omega \times [0,\,T]$ and $J \in H^6_{0,\,\sigma}((0,\,T);\,L^2(\Omega))$.  Let $x^\star \in \mathbb{R}^3 \setminus \Omega$ be fixed and $z\in \Omega$. Given the measurement $U(x^\star,\,t)$ for $t\in[\tilde{T},\,\tilde{T}+2\pi/b]$ with $\tilde{T}> T_J +c_0^{-1}|x^\star-z|$. We reconstruct the source $J(y,\,t)$ for $(y,\,t)\in \Omega\times (0,\,T)$ using the following steps:

\textbf{Step 1. Compute the field $V(z,\,t)$ for $t \in (0,\,T)$.}
  \begin{itemize}
    \item For $n=1,\,2,\,\ldots,\,N$, compute the coefficients $\alpha_n(x^\star,\,z),\,A_n(x^\star,\,z)$, and $B_n(x^\star,\,z)$ by \eqref{eq:alpha_n}, \eqref{eq:A_n for reconstruction}, and \eqref{eq:B_n for reconstruction}, respectively.

    \item For $n=1,\,2,\,\ldots,\,N$, compute the coefficients $C_n(x^\star,\,z)$ and $D_n(x^\star,\,z)$ by \eqref{eq:C_n for reconstruction} and \eqref{eq:D_n for reconstruction}, respectively.

    \item Approximate the field $V(z,\,t)$ by $V_N(z,\,t)$ in \eqref{eq:VN_truncation}.
    
   \end{itemize}

\textbf{Step 2. Compute the field $\partial_t^2 V_N(z,\,t)$ for $t \in (0,\,T)$.}
  \begin{itemize}
    \item Compute the field $V_N(z,\,t_l)$ for the data 
    \begin{equation}
    \{ V_N(z,\,t_l) | \ t_l=t+l \Delta \tau,\, l=0,\,\pm 1,\,\pm 2,\,\ldots,\, \pm n_t \},
    \label{eq:V_N for n_t}
    \end{equation}
    where $\Delta \tau>0$ and $n_t \in \mathbb{N}^+$.

    \item Compute the field $\partial_t^2 V_N(z,\,t)$ using a mollification method with the data \eqref{eq:V_N for n_t}.
  
    The mollification method used here is described as follows. Define
    \begin{equation}
    g(x):=\left\{
    \begin{aligned}
    & e^{\frac{1}{|x|^2-1}} & \quad & |x|<1,  \\
    & 0, & \quad & |x|\geq 1,
    \end{aligned}
    \right.
    \end{equation}
    and
    \begin{equation}
    \eta(x):=\frac{g(x)}{ \int_{\mathbb{R}} g(x)dx }, \quad \eta_\epsilon(x):=\frac{1}{\epsilon} \eta\Big(\frac{x}{\epsilon} \Big), \quad x\in \mathbb{R},
    \end{equation}
    where $\epsilon>0$ is a small parameter.
    Thus, $\eta_\epsilon(x)$ is a smooth function with $\mathrm{supp} \ \eta_\epsilon(x)= \{x \in \mathbb{R} | \ |x| \leq \epsilon \}$ and satisfies $\int_{\mathbb{R}} \eta_\epsilon(x) dx=1$.
    We approximate the function $f(x)$ and its derivative by
    \begin{equation}
    f_\epsilon(x) \approx \int_{\mathbb{R}} f(y)\eta_\epsilon(x-y) dy, \quad f_\epsilon^{\prime}(x)\approx \int_{\mathbb{R}} f(y)\eta_\epsilon^\prime(x-y) dy, \quad  x\in \mathbb{R}.
    \label{eq:f prime}
    \end{equation}
    In this context, the parameter satisfies $\epsilon :=(n_t+1) \Delta \tau$. By repeating the procedure in \eqref{eq:f prime}, we can obtain the second-order derivative numerically.
    \end{itemize}

 \textbf{Step 3. Reconstruct the source through $J(z,\,t) \approx c_0^{-2} \partial_t^2 V_N(z,\,t)-\Delta V_N(z,\,t)$.}
 
  In this step, we focus on the second-order partial derivatives with respect to the three spatial variables and repeat the differentiation procedure outlined in \textbf{Step 2}.
 \end{algorithm}

\begin{remark}
When using the mollification method to compute the numerical differentiation of $f$, the error estimate is given by \cite[Theorem 1.1]{Murio1993}, and it can be minimized by choosing $\epsilon=[2\delta_f/3M_2\sqrt{\pi}]^{1/2}$. Here $\delta_f$ is the noise level of $f$ and $M_2$ is the upper bound of $f^{\prime\prime}$. To compute the second-order derivative of $f$, we then apply the mollification method \eqref{eq:f prime} once more, using the updated error estimate $\delta_{f^\prime} \leq 2\pi^{-1/4} \sqrt{6M_2\delta_f}$. For simplicity, we just choose the same parameters in the mollification method of two stages. Therefore, it is reasonable to choose $\delta_f^{1/2} \lesssim \epsilon \lesssim \delta_f^{1/4} $ roughly. After selecting the parameter $\epsilon$, from \eqref{eq:f prime} and the support of $\eta_\epsilon(x)$, we choose the step $\Delta \tau$ in our simulation to perform the numerical integration using the trapezoidal rule over the interval $(-\epsilon,\,\epsilon)$. A more elaborate second-order numerical differentiation in one dimension can be performed when sufficient data are available and the noise is small; see \cite{X-L2010,W-L2006} for details.
\end{remark}

\section{Numerical experiments}
\label{sec:Numerical experiments}
In this section, we assess the validity of the asymptotic expansion of the wave field $W$ and examine the effectiveness of the source reconstruction method through three-dimensional numerical experiments.

\subsection{Asymptotic expansion of the wave field}

We recall that $W_N$ defined by \eqref{eq:W_N 1} is an approximation of $W$ by the asymptotic expansion. We numerically investigate the influence of the droplet size $a$, the truncation number $N$, and the temporal regularity of $J$ on $W_N$. For comparison, we denote by $W_{\mathrm{LSE}}$ the numerical simulation of $W$ obtained by solving the Lippmann-Schwinger equation.

First, we demonstrate the convergence behavior of $W_N$ as $N$ goes to infinity. From~\eqref{eq:W_expansion for inf}, the dominant part of $W$ consists of $\xi_n(x,\,t)$ and $\zeta_n(x,\,t)$ for $n\in \mathbb{N}^+$. For the nonvanishing $V$ and $W$, from \eqref{eq:xi_n}, \eqref{eq:zeta_n}, and the argument presented in Theorem~\ref{the:convergence of W}, we have
\begin{equation}
    |\xi_n(x,\,t)+\zeta_n(x,\,t)| \sim a \big(n-\frac{1}{2} \big)^2 \Big(\int_{B(0,\,1)} e_n(y) dy\Big)^2.
\end{equation}

Figure~\ref{fig:e_n} illustrates the decay behavior of the quantity
$(n-1/2)^{2} (\int_{B(0,1)} e_n(y) dy)^2$
with respect to $n$. As shown in the left panel, this quantity decays rapidly as $n$ increases, and the first four terms are significantly larger than the subsequent terms. In the right panel, it is demonstrated that the sum $\sum_{n=1}^{N} (n-1/2)^{2} (\int_{B(0,1)} e_n(y) dy)^2$ converges rapidly, and it is evident that $N\geq 4$ is sufficiently large to ensure a small truncation error.

\begin{figure}[htbp]
   \centering
   \includegraphics[width=12cm, angle=0]{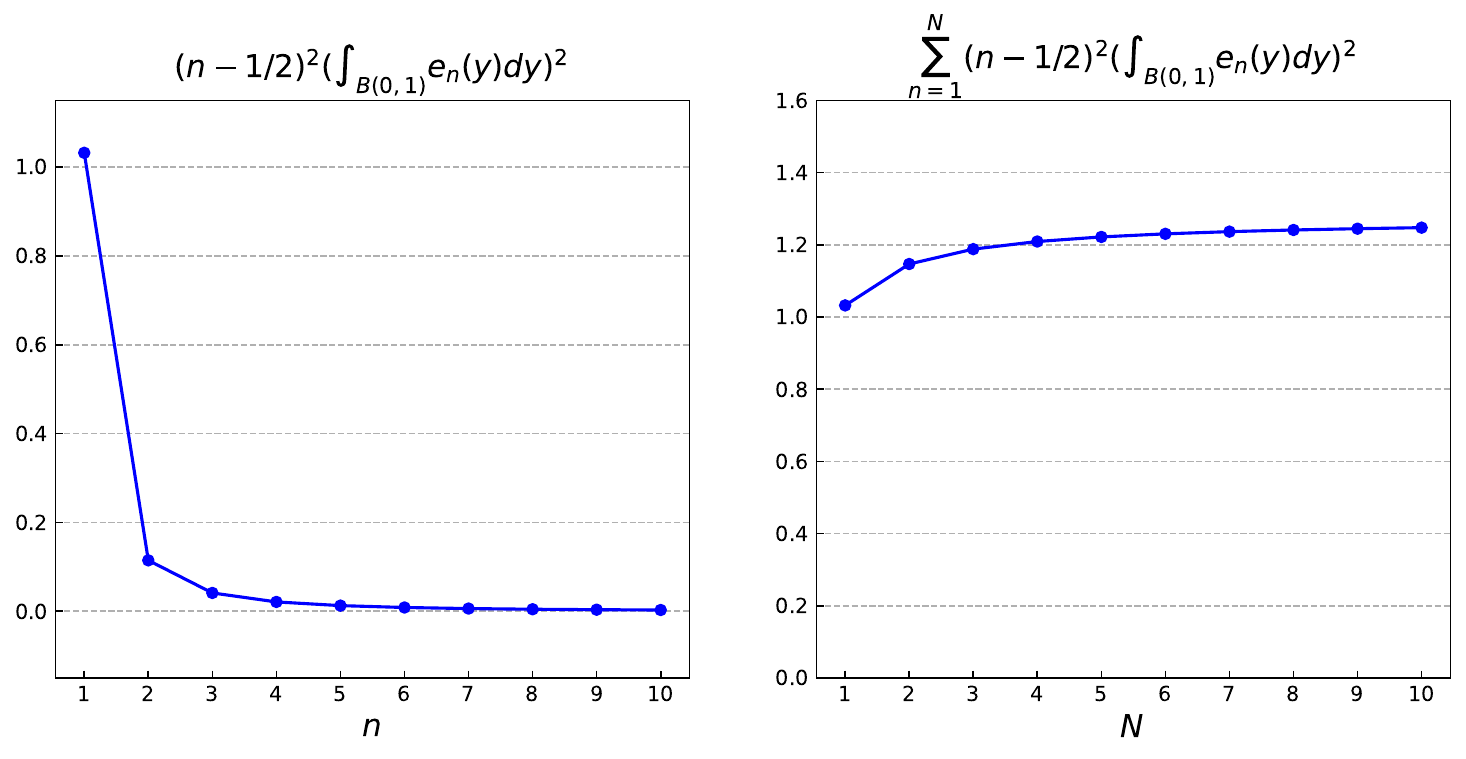}
   \caption{The decay behavior of $(n-1/2)^{2} (\int_{B(0,1)} e_n(y) dy)^2$ with respect to $n$.}
\label{fig:e_n}
\end{figure}

Next, we investigate the impact of the temporal regularity of $J$ and the droplet size $a$ on $W_N$.

\begin{example}
\label{exa:temporal regularity on W}
Consider a spherical domain $\Omega=B(0,\,1)$. Specifically, we define the causal function
\begin{equation}
V(x,\,t):= t^p \cdot (e^{r^2}+3x_2+x_3) H(t), \quad (x,\,t)\in \Omega\times(0,\,T), \quad p=3,\,4,
\label{eq:Vxt_example2}
\end{equation}
where $x=(x_1,\,x_2,\,x_3)$, $r=|x|$, and $H$ is the Heaviside function. Equivalently, the corresponding source is given by
\begin{equation*}
J(x,\,t)=c_0^{-2}V_{tt}(x,\,t)-\Delta V(x,\,t), \quad (x,\,t)\in \Omega\times(0,\,T).
\end{equation*}
\end{example}

In the numerical implementations, we take $T=4$, $q=2$, $\Delta t=0.1$, $c_0=1$, $b=4\pi$, $n_0=2$, $N_r=15$, and $N_s=12$. The droplet $D$ is injected at $z=(-0.2,\,0,\,0) \in \Omega$. 

\begin{figure}[htbp]
   \centering
   \includegraphics[width=12cm, angle=0]{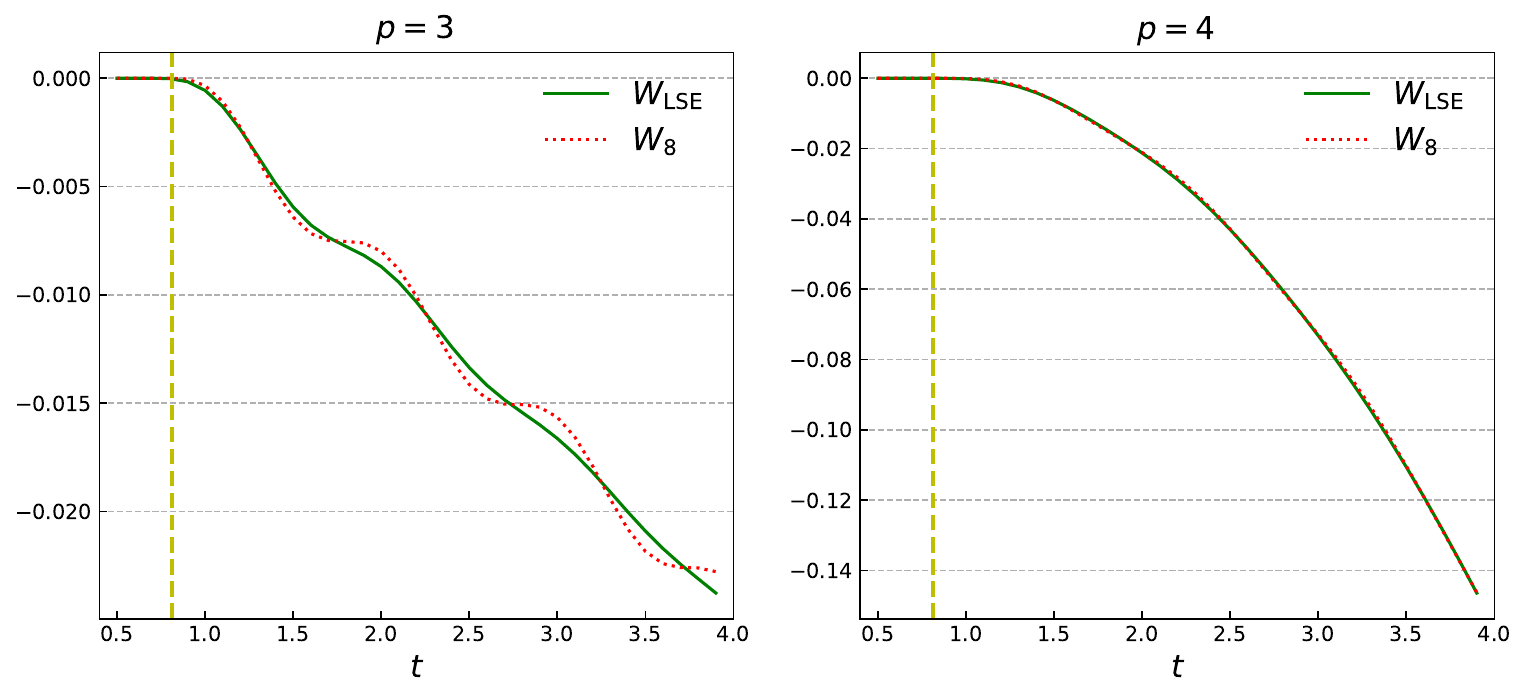}
   \caption{The wave field $W(x,\,t)$ evaluated at $x=(0.3,\,0.4,\,0.5)$, for sources with different temporal regularities as described in Example~\ref{exa:temporal regularity on W}. The left and right panels display the results for $p=3$ and $p=4$ in \eqref{eq:Vxt_example2}, respectively. The green solid line and the red dashed line represent the simulation results of $W_{\mathrm{LSE}}$ and $W_N$, respectively. The yellow dashed line indicates $t=c_0^{-1}|x-z|$, i.e., the first arrival time of the wave emitted from the source point $z$ to the point $x$.}
\label{fig:t_regular}
\end{figure}

Figure \ref{fig:t_regular} displays the simulation results of $W_{\mathrm{LSE}}(x,\,t)$ and $W_N$ for different $p$, with fixed $a=0.05$ and $N=8$. Note that $W(x,\,t)$ vanishes for $t<c_0^{-1}|x-z|$, which follows from the equations~\eqref{eq:Lippmann-Schwinger_equation 1} and~\eqref{eq:W_N 1}, since both $V$ and $J$ are causal. The time $t=c_0^{-1}|x-z|$ indicates the earliest arrival of a wave propagating from the source point $z$ to the point $x$. Clearly, the discrepancy between $W_{\mathrm{LSE}}(x,\,t)$ and $W_N$ decreases as the temporal regularity of the source increases. This behavior can be attributed to the use of the convolution spline method in solving the Lippmann–Schwinger equation, as this approach is particularly sensitive to the temporal regularity of the associated field.

Figure \ref{fig:W_p4} illustrates the simulation results of $W_{\mathrm{LSE}}(x,\,t)$ and $W_N$ for different $a$ and $N$, with fixed $p=4$. Notably, when $N\geq4$, increasing $N$ yields negligible improvement in accuracy. We observe that $W$ scales proportionally with the droplet size $a$, consistent with the analysis in Remark~\ref{rem:scale}. Furthermore, the difference $W_{\mathrm{LSE}}-W_N$ for $N=2,\,4,\,8$ is $O(a^2)$, consistent with Theorem \ref{the:W_expansion}. Moreover, the field $W_N(x,\,t)$ at $x=(0.3,\,0.4,\,0.5),\, t=3.5$ converges rapidly as $N$ increases, and its convergence behavior is analogous to $-\sum_{n=1}^{N} (n-1/2)^{2} (\int_{B(0,1)} e_n(y) dy)^2$.

\begin{figure}[htbp]
  \centering
	\begin{minipage}{1.0\linewidth}
		\vspace{3pt}
		\centerline{\includegraphics[width=12cm]{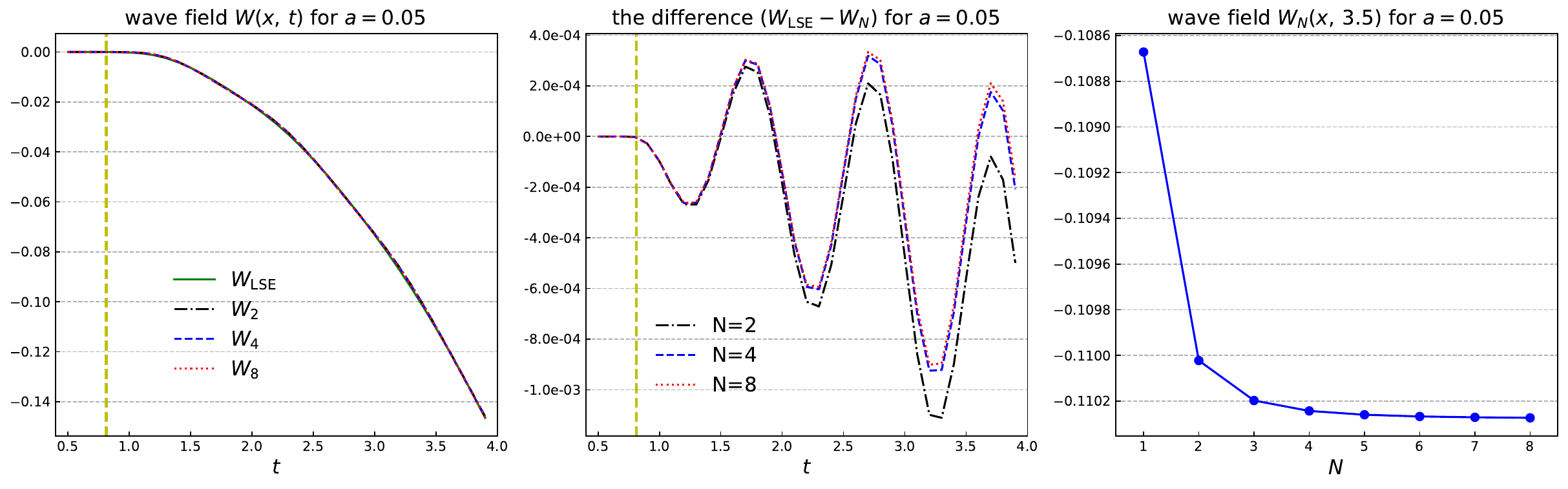}}
			\vspace{0.2cm}
		\centerline{\includegraphics[width=12cm]{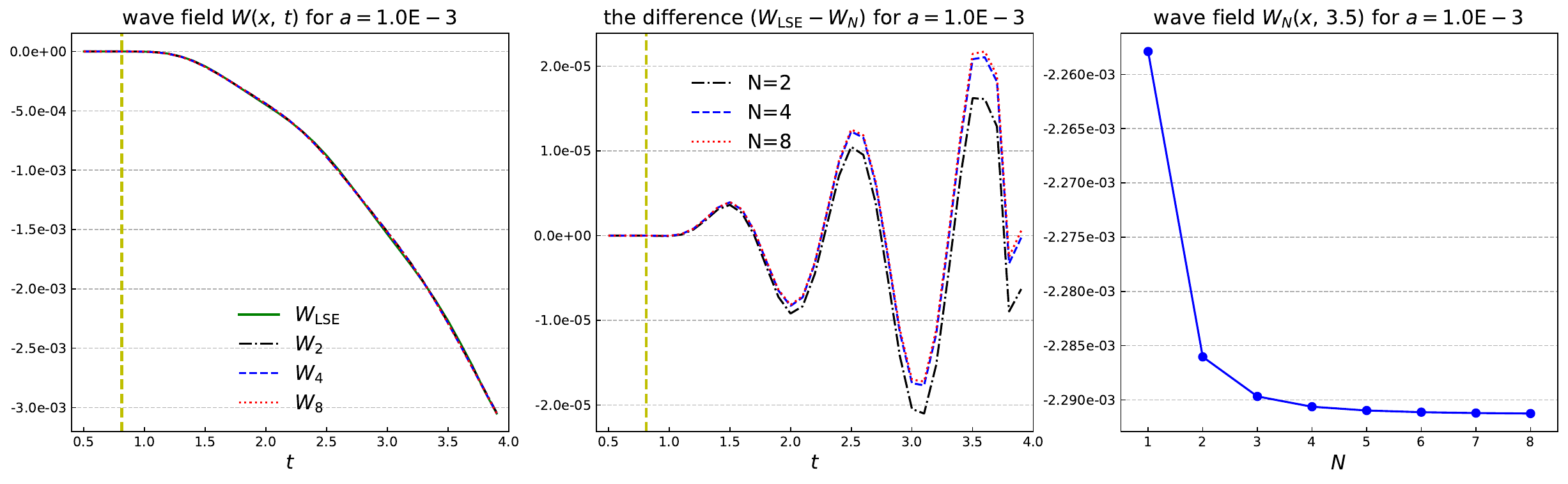}}
	\end{minipage}
	\caption{The numerical results of $W(x,\,t)$ at $x=(0.3,\,0.4,\,0.5)$ for Example \ref{exa:temporal regularity on W} with $p=4$. In the left panel, the green line represents the simulation result of $W_{\rm LSE}$. Meanwhile, the black, blue and red dashed lines represent the simulation results of $W_N$ for $N=2,\,4,\,8$, respectively. The yellow dashed line indicates $t=c_0^{-1}|x-z|$. The middle panel shows the differences $W_{\mathrm{LSE}} -W_N$ for $N=2,\,4,\,8$, respectively. And the right panel presents the wave field $W_N(x,\,t)$ at $x=(0.3,\,0.4,\,0.5),\, t=3.5$ for $N=1,\,2,\,\ldots,\,8$. The top and bottom panels present the results for $a=0.05$ and $a=1.0\mathrm{E}-3$, respectively.}
\label{fig:W_p4}
\end{figure}

From Figures~\ref{fig:t_regular} and~\ref{fig:W_p4}, we conclude that although Theorem~\ref{the:W_expansion} requires the source $J \in H^6_{0,\,\sigma}((0,\,T);\,L^2(\Omega))$, our method performs well for sources with lower temporal regularity. 

\subsection{Reconstruction of the source}
In this subsection, we reconstruct the source based on Theorem~\ref{the:inverse sourece} in theory and Algorithm~\ref{alg:invese_source} for its numerical implementation. We compare the reconstruction results with the exact values of $V,\,V_{tt}$ and $J$.

\begin{example}
\label{exa:rec-J}
Consider a cubic domain $\Omega=(-0.25,\,0.25)\times(-0.25,\,0.25)\times(-0.25,\,0.25)$. Define
\begin{equation*}
V(x,\,t):=\left\{
\begin{aligned}
& 10 \Big( \sin^3(t)\cdot\sin^2 \big(2(t-T) \big) \Big) \cdot (0.16-r^2) \cdot (e^{0.16-r^2}+2x_2+x_3),  \quad  (x,\,t)\in \Omega\times(0,\,T), \\
& 0, \quad  (x,\,t)\in \Omega\times(T,\,+\infty), 
\end{aligned}
\right.
\end{equation*}
where $x=(x_1,\,x_2,\,x_3)$ and $r=|x|$.
Equivalently, the corresponding source is determined by
\begin{equation*}
J(x,\,t)=c_0^{-2}V_{tt}(x,\,t)-\Delta V(x,\,t), \quad (x,\,t)\in \Omega\times(0,\,T).
\end{equation*}
Let $U(x^\star,t)$ denote the field after injecting the droplet $D$ into $\Omega$, measured at a fixed point $x^\star=(1.2,\,0,\,0) \in \mathbb{R}^3\setminus \Omega$ over the time interval $[\tilde{T},\,\tilde{T}+2\pi/b]$. Here, we consider the noisy measurement data
\begin{equation*}
U^{\delta}(x^\star,\,t):=U(x^\star,\,t)(1+ \overline{\delta} \eta),
\end{equation*}
where $\overline{\delta}$ is the relative noise level and $\eta$ is the uniform distribution on $(-1,\,1)$.
\end{example}

Take $T=1,\,c_0=1,\,b=2\pi,\,\tilde{T}=3.1$, and $\overline{\delta} =a$. We set the truncation number $N=20$ both for simulating the field $U(x^\star,t)$ via the expansion \eqref{eq:W_expansion for inf} and for reconstructing the field $V$. Following Algorithm~\ref{alg:invese_source}, we first reconstruct the field $V$ over the domain $\Omega^{\prime}:=(-0.12,\,0.12)\times (-0.25,\,0.25)\times (-0.25,\,0.25)$, using a uniform spatial step size $\Delta x= 0.01$ along all three axes. Based on the reconstructed field $V$ on this fixed spatial grid, we approximate the Laplacian $\Delta V$ and the second-order time derivative $V_{tt}$ via spatial and temporal numerical differentiation, respectively. These derivatives are computed using the mollification method, with a step size $\Delta\tau=\Delta x$, which is convenient given the uniform spatial sampling of the data. Moreover, we choose the mollification parameter $\epsilon=0.07$ for a relative noise level $\overline{\delta}=1.0\mathrm{E}-3$, and $\epsilon=0.04$ for $\overline{\delta}=1.0\mathrm{E}-4$. It should be noted that, due to the mollification process, the effective domain over which the source is reconstructed becomes slightly smaller.

\begin{figure}[htbp]
  \centering
	\begin{minipage}{1.0\linewidth}
		\vspace{3pt}
		\centerline{\includegraphics[width=12cm]{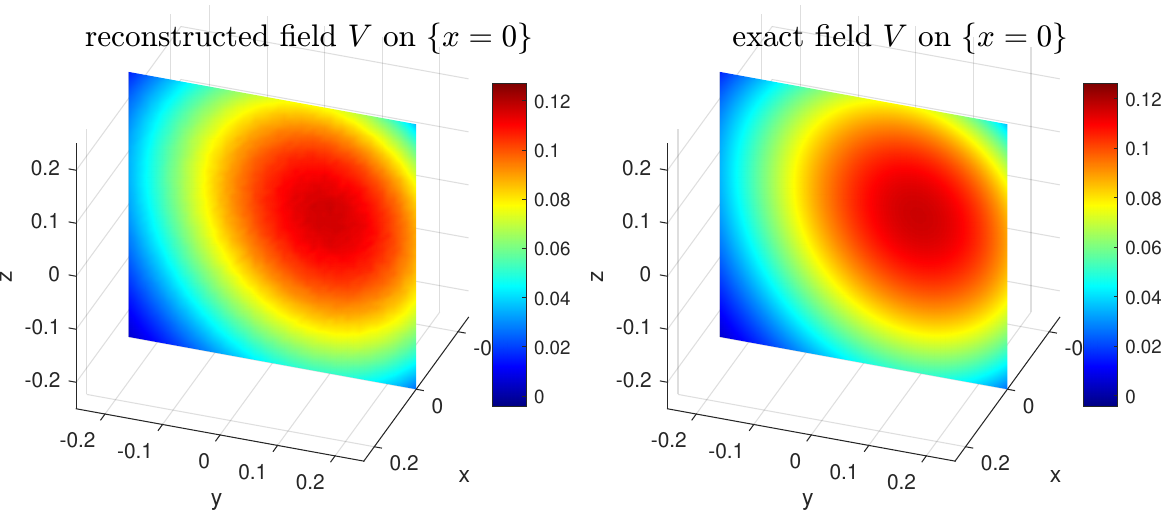}}
			\vspace{0.2cm}
            \centerline{\includegraphics[width=12cm]{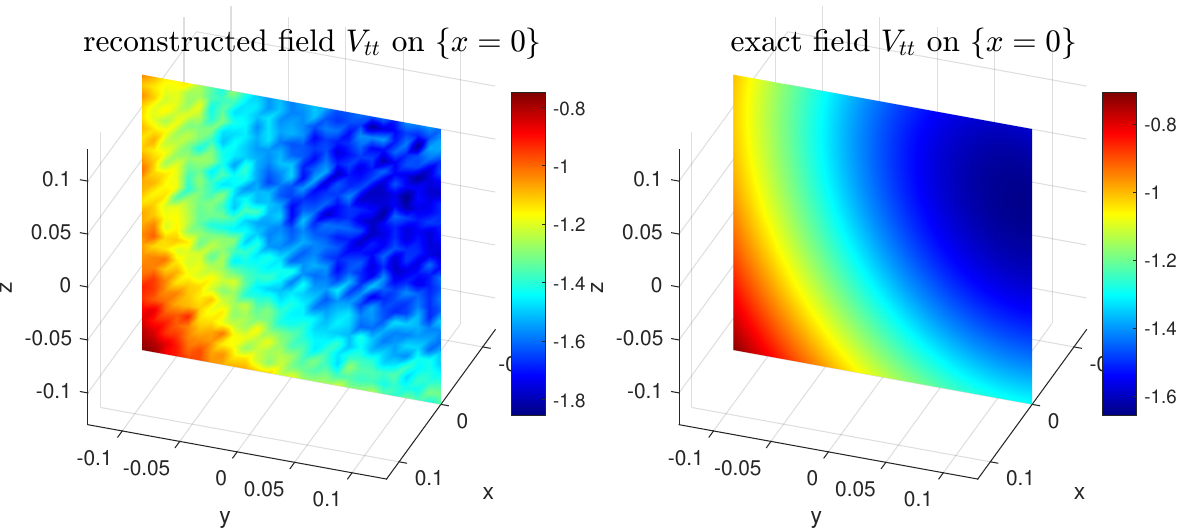}}
			\vspace{0.2cm}
		\centerline{\includegraphics[width=12cm]{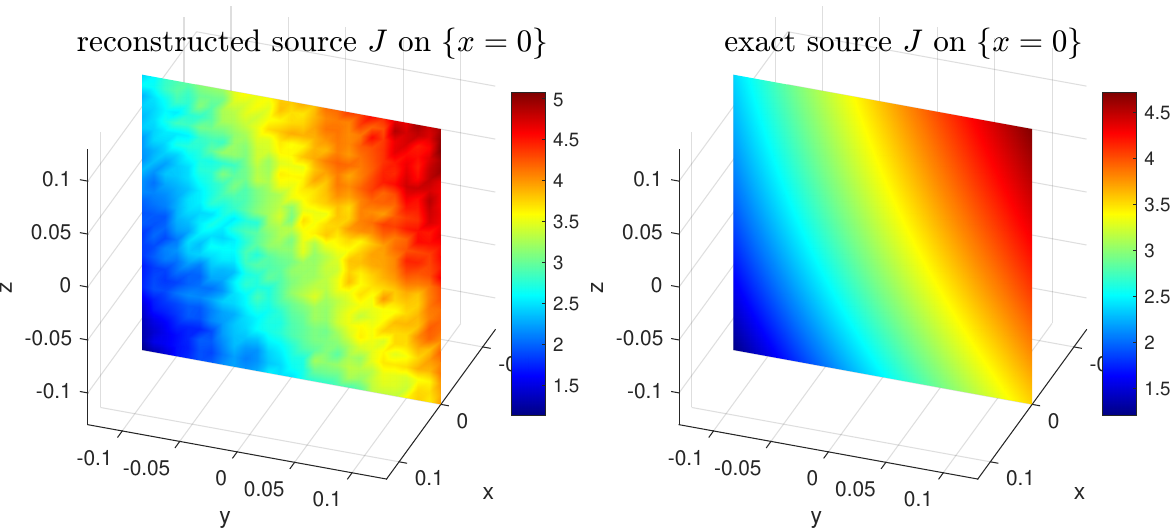}}
	\end{minipage}
	\caption{Comparison between the reconstructions and the exact ones at $t=0.8$ for Example \ref{exa:rec-J} with $a=\overline{\delta}=1.0\mathrm{E}-3$. The top, medium and bottom figures show the results of $V,\,V_{tt}$ and $J$, respectively. The left figures display the reconstructions and the right figures show the exact ones. In the top panel, $V$ is evaluated at $2601$ points on the slice $\{x=0\}$. On the snapshot of $V$, the reconstruction is performed in $(-0.25,\,0.25) \times (-0.25,\,0.25)$.  In the medium and bottom panels, $V_{tt}$ and $J$ are evaluated at $729$ points on the slice $\{x=0\}$. On the snapshots of $V_{tt}$ and $J$, the reconstructions are performed in $(-0.13,\,0.13) \times (-0.13,\,0.13)$.}
\label{fig:J_rec_a1e3}
\end{figure}

Figure \ref{fig:J_rec_a1e3} presents a snapshot of the reconstructed and exact fields for $a=1.0\mathrm{E}-3$ on the slice $\{x=0\}$ at $t=0.8$. The functions $V,\,V_{tt}$, and $J$ are reconstructed with the relative errors of $0.412\%,\,6.390\%$ and $11.706\%$ in the $L^2$-norm, respectively. Figure \ref{fig:J_rec_a1e4} presents a snapshot of the reconstructed and exact fields for $a=1.0\mathrm{E}-4$ on the slice $\{x=0\}$ at $t=0.8$. The functions $V,\,V_{tt}$, and $J$ are reconstructed with the relative errors of $0.040\%,\,5.244\%$ and $1.384\%$ in the $L^2$-norm, respectively. The relative errors of the reconstructions for $a=1.0\mathrm{E}-4$ are smaller than those for $a=1.0\mathrm{E}-3$. For both cases, the profiles of the reconstructions of $V,\,V_{tt}$, and $J$ are consistent with the exact ones.

For a higher noise level, such as $\overline{\delta}=1.0\mathrm{E}-2$, the field $V$ can still be reconstructed with a tolerable relative error of $4.27\%$ on the slice $\{x=0\}$. However, the accuracy of the second-order numerical differentiation and subsequent source reconstruction degrades significantly under this noise level. Using $\Delta\tau=0.02$ and $\epsilon=0.12$, the reconstructed source $J$ attains a relative $L^2$-norm error of $25.186\%$. Attempts to improve the reconstruction accuracy by adjusting these numerical differentiation parameters do not yield a noticeable reduction in error.

\begin{figure}[htbp]
  \centering
	\begin{minipage}{1.0\linewidth}
		\vspace{3pt}
		\centerline{\includegraphics[width=12cm]{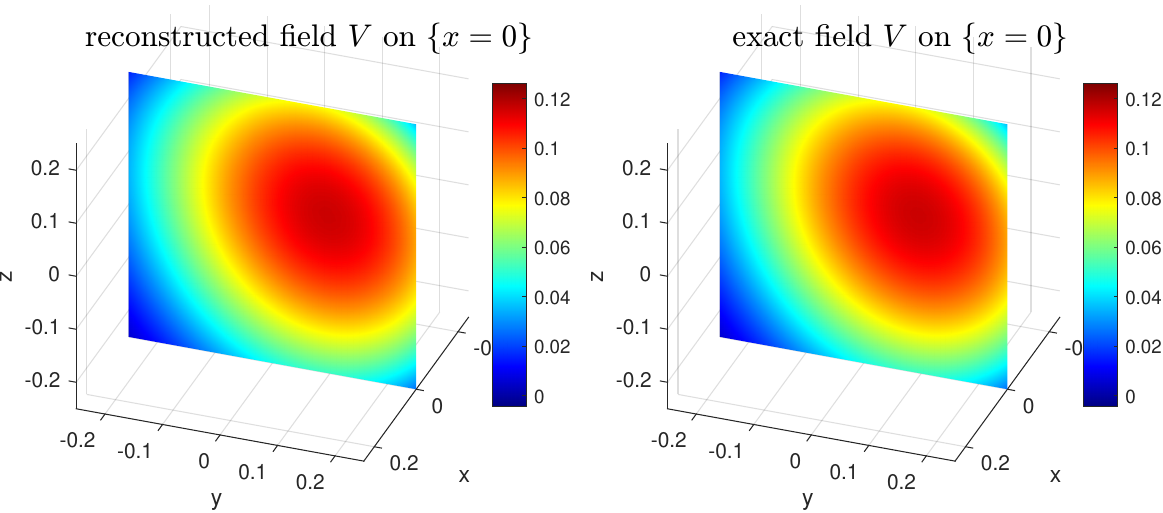}}
			\vspace{0.2cm}
            \centerline{\includegraphics[width=12cm]{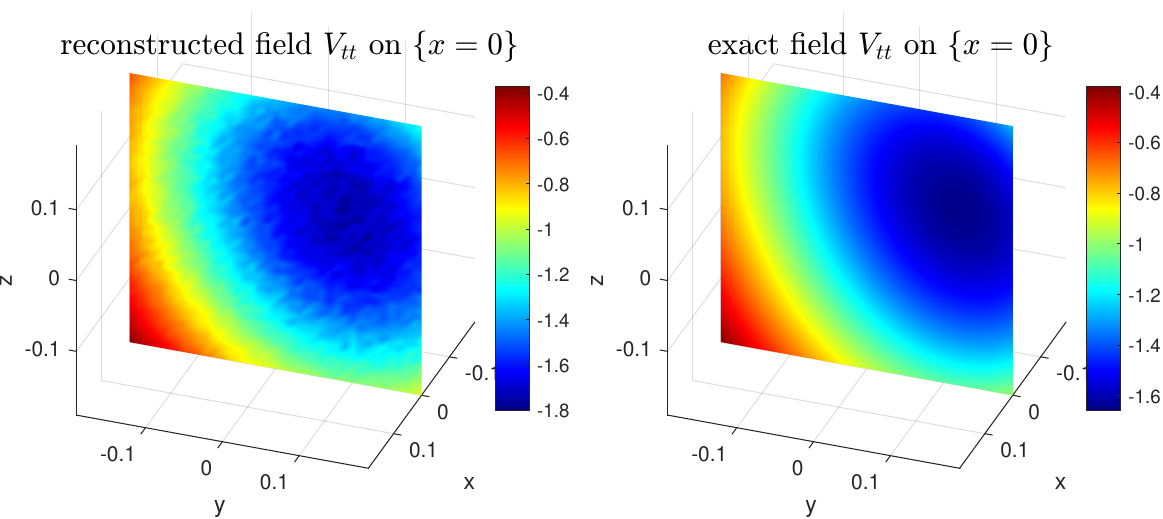}}
			\vspace{0.2cm}
		\centerline{\includegraphics[width=12cm]{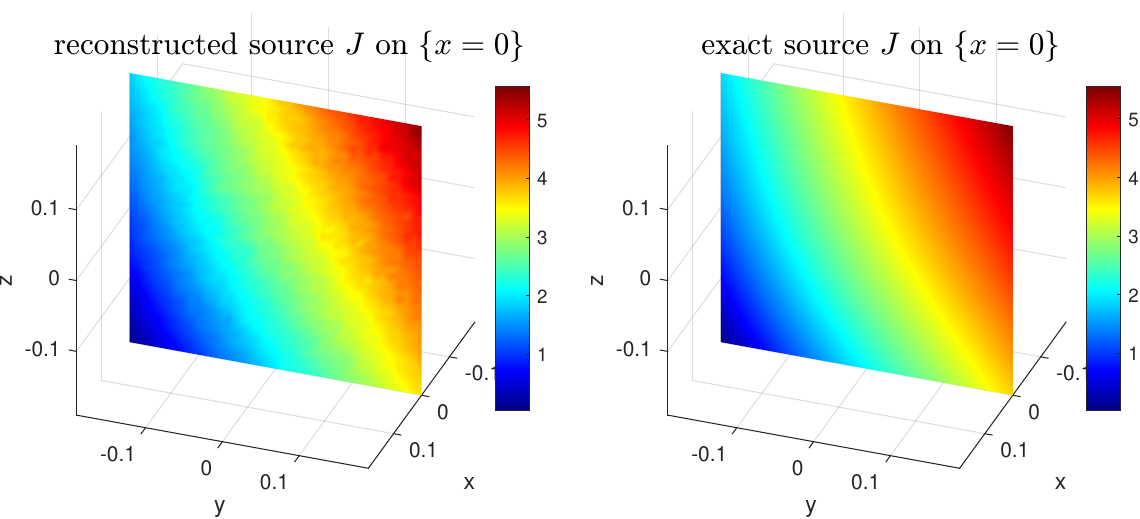}}
	\end{minipage}
	\caption{Comparison between the reconstructions and the exact ones at $t=0.8$ for Example \ref{exa:rec-J} with $a=\overline{\delta}=1.0\mathrm{E}-4$. The top, medium and bottom figures show the results of $V,\,V_{tt}$ and $J$, respectively. The left figures display the reconstructions and the right figures show the exact ones. In the top panel, $V$ is evaluated at $2601$ points on the slice $\{x=0\}$. On the snapshot of $V$, the reconstruction is performed in $(-0.25,\,0.25) \times (-0.25,\,0.25)$. In the medium and bottom panels, $V_{tt}$ and $J$ are evaluated at $1521$ points on the slice $\{x=0\}$. On the snapshots of $V_{tt}$ and $J$, the reconstructions are performed in $(-0.19,\,0.19) \times (-0.19,\,0.19)$.}
\label{fig:J_rec_a1e4}
\end{figure}

\section{Conclusions}
\label{sec:Conclusions}
In this paper, we investigated the inverse source problem for the wave equation in three dimensions. Considering the ill-posedness and measurement noise in applications, we proposed a new approach to reconstruct the source based on the droplet-induced asymptotics. After injecting a small droplet with high contrast at $z\in\Omega$, we measure the wave field $U$ at a fixed point $x\notin \Omega$ over a time interval. Note that this wave field can be approximately represented by an infinite series in terms of the eigensystem of the Newtonian operator, containing the information on $V$ before the droplet injection. Then we reconstructed the wave field $V$ at $z$ using the Riesz theory. We derived error estimates for both the asymptotic expansion of $U$ and the reconstruction of $V$. Moreover, we provided a criterion for selecting the truncation number in relation to the droplet size. After the wave field $V$ was reconstructed, we applied the regularization techniques to numerically differentiate it and further reconstructed the source. Numerical experiments validated the effectiveness of the proposed method for solving the inverse source problem in three dimensions. As part of future work, we expect to study the simultaneous reconstruction of multiple unknown parameters in the wave equation, such as the wave speed and source function.

\appendix
\section*{Appendix}
\setcounter{equation}{0}
\renewcommand{\theequation}{A.\arabic{equation}}
\renewcommand{\thelemma}{A.\arabic{lemma}}
\renewcommand{\thetheorem}{A.\arabic{theorem}}

In this appendix, we investigate the eigensystem of the Newtonian operator over a three-dimensional spherical domain, i.e., $D=\{ x\subset \mathbb{R}^3: |x|<a \}$, which is utilized in the asymptotic expansion of $W$.

Note that the Newtonian potential given by
\begin{equation}
u(x)=\mathcal N (f) = \int_D \Phi_0(x,\,y)f(y) dy, \quad x\in D
\label{eq:u_f}
\end{equation}
satisfies
\begin{equation*}
-\Delta u (x)=f(x), \quad x\in D,
\end{equation*}
where 
\begin{equation*}
\Phi_0(x,\,y):=\frac{1}{4\pi |x-y|}, \quad x\neq y,
\end{equation*}
is the fundamental solution to the Laplace equation. Consider the associated spectral problem
\begin{equation}
\mathcal{N}(u)(x)=\lambda u(x), \quad x\in D.
\label{eq:eigen system}
\end{equation}

To begin with, we present the boundary properties of the Newtonian potential as described in \cite{K-S2011}.

\begin{lemma}
\label{lem:BC-potential}
For any function $f\in L^2(D)$, the Newtonian potential \eqref{eq:u_f} satisfies the boundary condition
\begin{equation}
-\frac{u(x)}{2}-\int_{\partial D}\Phi_0(x,\,y) \frac{\partial u}{\partial \nu}(y) ds(y)
+ \int_{\partial D} \frac{\partial \Phi_0(x,\,y)}{\partial \nu(y)} u(y) ds(y)=0, \quad x\in \partial D,
\label{eq:BC-potential}
\end{equation}
where $\frac{\partial}{\partial \nu}$ denotes the outer normal derivative on the boundary.
Conversely, if a function $u\in H^2(D)$ satisfies the Poisson equation
\begin{equation}
-\Delta u(x)=f(x), \quad x\in D,
\end{equation}
and the boundary condition \eqref{eq:BC-potential}, then it determines the Newtonian potential \eqref{eq:u_f}.
\end{lemma}

To compute the boundary integral involving spherical harmonics, we state the following result from \cite[Section 3.6]{C-K2019}.
\begin{lemma}
\label{lem:integral for phi Y}
For the spherical harmonics $Y_n$ of order $n$, we find
\begin{equation}
\int_{\mathbb{S}^2} \Phi_0(\hat{x},\,y) Y_n(y) ds(y)=\frac{1}{2n+1} Y_n(\hat{x}), \quad \hat{x}\in \mathbb{S}^2,
\label{eq:integral for phi Y}
\end{equation}
where $\mathbb{S}^2$ denotes the surface of the unit sphere.
\end{lemma}

We now present another lemma related to the integral of spherical harmonics.
\begin{lemma}
\label{lem:integral for dp Y}
For the spherical harmonics $Y_n$ of order $n$, we have
\begin{equation}
\int_{\mathbb{S}^2} \frac{\partial \Phi_0(\hat{x},\,y)}{\partial \nu(y)} Y_n(y) ds(y)=\Big( \frac{n}{2n+1} -\frac{1}{2} \Big) Y_n(\hat{x}), \quad \hat{x}\in \mathbb{S}^2.
\label{eq:integral for dp Y}
\end{equation}
\end{lemma}

\begin{proof}
Consider the entire solution to the Helmholtz equation in $\mathbb{R}^3\setminus \{0\} $
\begin{equation}
u_n(x):= j_n(k|x|)Y_n(\hat{x}), \quad \hat{x}=\frac{x}{|x|},
\label{eq:u_n}
\end{equation}
where  $k$ is the frequency and the functions $j_n$'s are spherical Bessel functions given by
\begin{equation}
j_n(t):=\sum_{p=0}^\infty \frac{(-1)^p t^{n+2p}}{2^p p! 1\cdot 3 \cdots (2n+2p+1) }.
\label{eq:spherical Bessel}
\end{equation}

Applying Green's theorem \cite[Theorem 2.5]{C-K2019} to $u_n$ on a sphere of radius $r$, we obtain
\begin{equation}
\int_{|y|=r} \Big\{ u_n(y) \frac{\partial \Phi(x,\,y)}{\partial \nu(y)} - \frac{\partial u_n}{\partial \nu} (y) \Phi(x,\,y) \Big\} ds(y) =0.
\quad |x|>r,
\label{eq:Green-u}
\end{equation}
where 
\begin{equation*}
\Phi(x,\,y):=\frac{e^{ik |x-y|}}{4\pi |x-y|},\quad x\neq y,
\end{equation*}
is the fundamental solution to the Helmholtz equation. Noting that on $|y|=r$,
\begin{equation}
u_n(y)=j_n(kr) Y_n(\hat{y}), \quad \frac{\partial u_n}{\partial \nu} (y)=k j_n^\prime(kr)Y_n(\hat{y}),\quad \hat{y}=\frac{y}{|y|},
\end{equation}
and from \eqref{eq:Green-u}, we have
\begin{equation}
\int_{|y|=r}  \frac{\partial \Phi(x,\,y)}{\partial \nu(y)} Y_n(\hat{y})  ds(\hat{y}) = \frac{k j_n^\prime(kr)}{j_n(kr)}  \int_{|y|=r} \Phi(x,\,y) Y_n(\hat{y}) ds(y) .
\label{eq:sd_relation1}
\end{equation}
Considering $r=1$ and passing $|x|\to r$ from the exterior, we get
\begin{equation}
\int_{\mathbb{S}^2}  \frac{\partial \Phi(\hat{x},\,y)}{\partial \nu(y)} Y_n(y)  ds(y) + \frac{1}{2}Y_n(\hat{x})  = \frac{k j_n^\prime(k)}{j_n(k)}  \int_{\mathbb{S}^2} \Phi(\hat{x},\,y) Y_n(y) ds(y) , \quad \hat{x}\in \mathbb{S}^2.
\label{eq:sd_relation2}
\end{equation}
Here, we have utilized the jump of the double layer potential. By passing to the limit $k\to 0$, we observe from \eqref{eq:spherical Bessel} that
\begin{equation}
\lim _{k\to 0} \frac{k j_n^\prime(k)}{j_n(k)}=n,
\end{equation}
and the fundamental solution $\Phi(x,\,y)$ reduces to $\Phi_0(x,\,y)$. From \eqref{eq:sd_relation2} and Lemma \ref{lem:integral for phi Y}, we obtain \eqref{eq:integral for dp Y}.
\end{proof}

Now we reformulate the eigensystem of the Newtonian operator, following the approach outlined in \cite[Theorem 4.2]{K-S2011} and applying Lemmas \ref{lem:integral for phi Y} and \ref{lem:integral for dp Y} to \eqref{eq:BC-potential}.  

The eigenvalues $\lambda_{lj}$ are given by
\begin{equation}
\lambda_{lj}:=\frac{a^2}{[\mu_j^{(l+\frac{1}{2})}]^2},\quad l=0,\,1,\ldots, \ j=1,\,2,\,\ldots,
\label{eq:lambda-nj}
\end{equation}
where the $\mu_j^{(l+\frac{1}{2})}$ are the the roots of the equation
\begin{equation}
J_{l-\frac{1}{2}} (\mu_j^{(l+\frac{1}{2})}) =0.
\label{eq:eigen_value}
\end{equation}
The eigenfunctions corresponding to each eigenvalue $\lambda_{lj}$ can be represented in the form
\begin{equation}
u_{ljm} = \frac{1}{\sqrt{r}} J_{l+\frac{1}{2}} \big( r/\sqrt{\lambda_{lj}} \big) Y_l^m(\theta,\,\varphi), \quad m=-l,\,\ldots,\,0,\,\ldots,\,l,
\label{eq:eigen function}
\end{equation}
where $J_{l+\frac{1}{2}}$ are the Bessel functions and $Y_l^m(\theta,\,\varphi)$ are spherical functions.

Specially, for the case $l=0$, we have $m=0$. Note that
\begin{equation*}
J_{-\frac{1}{2}}(y)= \sqrt{\frac{2}{\pi y}} \cos(y), \quad y>0.
\end{equation*}
From \eqref{eq:eigen_value}, we have
\begin{equation}
\mu_j^{(\frac{1}{2})} = (j-\frac{1}{2} )\pi ,\quad j=1,\,2,\,\ldots
\label{eq:eigenvalue_nonvanish}
\end{equation}
and
\begin{equation}
u_{0j0}(x)= \frac{1}{r} \sqrt{ 2 a / \big(\pi \mu_j^{(\frac{1}{2})} \big) } \sin \Big( \frac{\mu_j^{(\frac{1}{2})} r}{a} \Big),
\quad r=|x|>0, \quad j=1,\,2,\,\ldots,
\label{eq:eigenfunction_nonvanish}
\end{equation}
which are utilized in the asymptotic expansion in Theorem~\ref{the:W_expansion}, with nonvanishing averages.

Note that
\begin{equation}
\begin{aligned}
\mathcal{N}(u_{ljm})(x) &= \int_D \Phi_0(x,\,y) u_{ljm}(y) dy \\
& =\int_0^a \int_0^\pi \int_0^{2\pi} \frac{1}{4\pi \big| x-y(r,\,\theta,\,\varphi) \big|}  \Big( \frac{1}{\sqrt{r}}J_{l+\frac{1}{2}} \big(r/\sqrt{\lambda_{lj}} \big) Y_l^m(\theta,\,\varphi) \Big) r^2\sin(\theta) dr d\theta d\varphi.
\label{eq:N_u integral}
\end{aligned}
\end{equation}
We redefine the eigenfunctions at $r=0$ as
\begin{equation}
u_{ljm} |_{r=0} =\left\{
\begin{aligned}
 & \sqrt{\frac{2 \mu_j^{(\frac{1}{2})} }{\pi a}}, & \quad l=0, \\
 & 0, & \quad l>0.
\end{aligned}
\right.
\end{equation}
Therefore, there is only one singularity point $x$ in the integral \eqref{eq:N_u integral}. For the numerical computation of \eqref{eq:N_u integral}, see \eqref{eq:volume integration 2}.

From \eqref{eq:eigen system}, we define the error as
\begin{equation}
\mathrm{\textbf{err}}:= \Big( \sum_{i=1}^{\overline{N}} \Big| u(x_i) - \lambda^{-1} \mathcal{N}(u)(x_i) \Big| \Big)/ \overline{N},
\end{equation}
to evaluate the eigensystem represented by \eqref{eq:lambda-nj} and \eqref{eq:eigen function}. Consider a ball with radius $a=1$ and set $N_r=15,\,N_s=12$. The points are selected evenly within the interval $[-1,\,1]$ with a step size $\Delta x=2/19$ along all the three axes. This results in $\overline{N}=3112$ points within the ball $B(0,\,0.95)$. Table~\ref{tab:n_1} shows the errors for $j=1,\,2,\,\ldots,\,6$ with different $m$ and $l$. As $j$ increases, the eigenfunction becomes more oscillating, resulting in a larger error. In general, the error remains small, and the eigensystem is validated numerically.

\begin{table}[!htbp]
    \caption{$\mathrm{\textbf{err}}$ of the eigensystem}	\centering
	\begin{tabular}{c c c c c c c}
		\hline
		\  & $j=1$ & $j=2$ & $j=3$ & $j=4$ & $j=5$ & $j=6$ \\
		\hline
		$l=0,\,m=0$  & $3.137\mathrm{E}-16$ & $3.725\mathrm{E}-15$ & $1.905\mathrm{E}-11$ & $4.350\mathrm{E}-8$ & $1.024\mathrm{E}-5$ & $5.702\mathrm{E}-4$ \\
		$l=1,\,m=0$  & $4.627\mathrm{E}-16$ & $1.636\mathrm{E}-13$ & $1.482\mathrm{E}-9$ & $7.091\mathrm{E}-7$ & $6.195\mathrm{E}-5$ & $1.697\mathrm{E}-3$ \\
            $l=1,\,m=1$  & $3.881\mathrm{E}-16$ & $1.720\mathrm{E}-13$ & $1.534\mathrm{E}-9$ & $7.166\mathrm{E}-7$ & $6.043\mathrm{E}-5$ & $1.579\mathrm{E}-3$ \\
		\hline
	\end{tabular}
    \label{tab:n_1}
\end{table}

\bigskip
{\bf Acknowledgment:} The first and third authors were supported by the National Natural Science Foundation of China (Nos. 12471395, 12241102), the Jiangsu Provincial Scientific Research Center of Applied Mathematics under Grant No. BK20233002, and the Big Data Computing Center of Southeast University. The second author was supported by the Austrian Science Fund (FWF) grants P: 36942 and P: 32660.

\end{sloppypar}
\end{document}